\documentclass[hidelinks,onefignum,onetabnum]{siamart250211} 

\usepackage{amsfonts}  
\usepackage{amssymb}   

\usepackage{graphicx}
\usepackage[T1]{fontenc}

\usepackage{bm}
\usepackage{color}
\usepackage{xcolor}
\usepackage{xargs}
\usepackage{setspace}
\usepackage{epstopdf}
\usepackage[colorinlistoftodos,prependcaption,textsize=tiny]{todonotes}
\usepackage[font=small,format=plain,labelfont=bf,up,textfont=it,up]{caption}
\usepackage{subcaption}
\usepackage{tabularx}
\usepackage{multirow}
\usepackage{fancybox}
\usepackage{lscape}
\usepackage{listings}
\usepackage{float}
\usepackage{rotating}
\usepackage{url}
\usepackage{booktabs}
\usepackage{enumerate}
\usepackage{tcolorbox}
\usepackage[normalem]{ulem}

\newcommand{\term}[1]{\emph{#1}}

\newcommand{\mcg}{\mathcal{G}}
\newcommand{\mcf}{\mathcal{F}}
\newcommand{\mcd}{\mathcal{D}}
\newcommand{\Bx}{\mathbf{x}}
\newcommand{\Bu}{u}
\newcommand{\Bv}{v}

\newcommand{\sdtd}{\tilde{d}}

\newcommand{\paramMu}{\mu}
\newcommand{\paramLambda}{\lambda}
\newcommand{\paramEta}{\eta}

\newcommand{\pderiv}[2]{\frac{\partial #1}{\partial #2}}

\newcommand{\grad}{\nabla}
\newcommand{\laplace}{\Delta}

\newcommand{\normLtwo}[1]{\left\lVert#1\right\rVert_{L^2}}

\newsiamremark{remark}{Remark}
\newsiamremark{assumption}{Assumption}

\headers{Interpolated Discrepancy Data Assimilation}{T. Wu, H. Godinez, V. Gyrya, J. M. Hyman}

\title{Interpolated Discrepancy Data Assimilation for PDEs with Sparse Observations\thanks{Submitted to the editors DATE.}}

\author{Tong Wu\thanks{University of Texas San Antonio, San Antonio, TX  (\email{tong.wu@utsa.edu}).}
\and Humberto Godinez\thanks{Los Alamos National Laboratory, Los Alamos, NM (\email{hgodinez@lanl.gov}).}
\and Vitaliy Gyrya\thanks{Los Alamos National Laboratory, Los Alamos, NM (\email{vitaliy\_gyrya@lanl.gov}).}
\and James M. Hyman\thanks{Tulane University, New Orleans, LA (\email{mhyman@tulane.edu}).}}

\begin{document}

\maketitle

\begin{abstract}
Sparse sensor networks in weather and ocean modeling observe only a small fraction of the system state, which destabilizes standard nudging-based data assimilation. We introduce Interpolated Discrepancy Data Assimilation (IDDA), which modifies how discrepancies enter the governing equations. Rather than adding observations as a forcing term alone, IDDA also adjusts the nonlinear operator using interpolated observational information. This structural change suppresses error amplification when nonlinear effects dominate. We prove exponential convergence under explicit conditions linking error decay to observation spacing, nudging strength, and diffusion coefficient. The key requirement establishes bounds on nudging strength relative to observation spacing and diffusion, giving practitioners a clear operating window. When observations resolve the relevant scales, error decays at a user-specified rate. Critically, the error bound scales with the square of observation spacing rather than through hard-to-estimate nonlinear growth rates. We validate IDDA on Burgers flow, Kuramoto-Sivashinsky dynamics, and two-dimensional Navier-Stokes turbulence. Across these tests, IDDA reaches target accuracy faster than standard interpolated nudging, remains stable in chaotic regimes, avoids non-monotone transients, and requires minimal parameter tuning. Because IDDA uses standard explicit time integration, it fits readily into existing simulation pipelines without specialized solvers. These properties make IDDA a practical upgrade for operational systems constrained by sparse sensor coverage.
\end{abstract}



\section{Introduction}
\label{sec:introduction}

Accurate modeling of complex physical systems often relies on partial differential equations (PDEs) to describe their evolution. However, sparse, incomplete, and irregularly distributed observational data frequently limit practical utility. Weather station networks are typically 100--200 kilometers apart, while numerical weather models resolve features at 10--25 kilometers. Ocean moorings measure temperature and salinity in hundreds of locations over millions of square kilometers. Satellite observations provide broader coverage but measure indirect quantities at coarser temporal resolutions than model time steps. Continuous data assimilation methods must bridge these gaps by combining incomplete observations with physics-based models to produce reliable forecasts.

The challenge of sparse and irregularly distributed observations has received increasing attention across data assimilation methodologies (e.g., early nudging or Newtonian relaxation approaches~\cite{anthes1974data,TheInitialization}). Traditional variational methods~\cite{le1986variational,talagrand1987variational} and ensemble Kalman filtering~\cite{evensen1994sequential,evensen2009data} handle sparse data through background error covariance matrices, but require tuning many parameters and can be expensive for high-dimensional systems. Recent advances in continuous data assimilation~\cite{azouani2014continuous,azouani2014feedback,foias2016discrete,farhat2017continuous,larios2024nonlinear} show promise for PDEs but mainly for observation-rich settings. Standard nudging techniques~\cite{TheInitialization,lewis2012,asch2016data,kalnay2003,law2015} operate on physical-space data, but often exhibit slow convergence with fixed nudging parameters. Preliminary comparative studies suggest that continuous data assimilation approaches proposed by Azouani, Olson, and Titi (AOT) can offer computational advantages over ensemble Kalman filtering while maintaining comparable accuracy~\cite{ning2024assessment}. Alternative approaches using machine learning for extremely sparse observations have emerged, but require extensive training data and computational resources that may not be available operationally.

The AOT framework introduced a feedback control mechanism that nudges model trajectories toward observations and proved convergence for a broad class of dissipative systems~\cite{cao2001numerical,olson2003determining}. Foias et al.~\cite{foias2016discrete} developed a related scheme for Navier--Stokes solutions and their statistics. Applying nudging with sparse physical-space data and fixed parameters can, however, yield suboptimal performance, especially for complex systems~\cite{ibdah2020fully,garcia2020uniform,gesho2016computational}. Recent studies~\cite{titi2023,titi2024inadequacy} explored enhancements including time-delay techniques and model error adjustments. Although these improve robustness, limitations persist with sparse data or high-dimensional systems.

Recent advances show that structural modifications to continuous data assimilation can achieve notable properties. Carlson, Larios, and Titi proved that nonlinear feedback enables superexponential convergence for two-dimensional Navier--Stokes equations~\cite{carlson2024super}. Similar nonlinear or hybrid nudging formulations have been investigated in low-dimensional and magnetohydrodynamic settings~\cite{du2021lorenz,hudson2019mhd}. Diegel et al.~\cite{diegel2025analysis} established that large nudging parameters with implicit integration algorithms can yield optimal accuracy. These methods represent significant theoretical breakthroughs but require nonlinear feedback or implicit schemes with associated computational overhead.

We propose \term{Interpolated Discrepancy Data Assimilation} (IDDA), extending the continuous data assimilation framework represented by the AOT approach~\cite{azouani2014continuous}. IDDA adapts AOT to address challenges posed by sparse observations in physical space. Many numerical simulations of AOT-based methods focused on Fourier-space observations, which may not be available in practice. IDDA operates directly on data at discrete locations, consistent with operational networks such as weather stations~\cite{TheInitialization,lewis2012}. 

The interpolated discrepancy is defined as the difference between the interpolation constructed from the observational data and that obtained from the assimilated model solution at the same spatial locations and time. In this work, the spatial interpolation operator is used to map both the true reference solution and the assimilated solution onto the same observation grid, producing two corresponding interpolations. The interpolated discrepancy is therefore the difference between these two interpolated representations, which reflects the model--data mismatch expressed on the observational scale. This quantity plays a central role in guiding model adjustment during assimilation. IDDA improves robustness by introducing the interpolated discrepancy not only as an additive feedback term, but also within the operator that governs the model dynamics. By modifying the evolution using information derived from the discrepancy on the observational scale, the method more accurately reproduces the true system dynamics and mitigates errors that arise when observations are sparse or the nudging parameter is small.

IDDA offers a complementary approach, achieving predictable exponential rates through linear nudging with standard explicit time stepping. This makes IDDA readily implementable within existing frameworks while delivering substantial improvements over standard interpolated AOT. IDDA is designed for dissipative systems with sufficient regularization. Recent work demonstrates that nudging algorithms are inadequate for non-dissipative systems lacking determining modes~\cite{titi2024inadequacy}. Our focus on equations with positive diffusion (Burgers', Navier--Stokes) or sufficient hyperdiffusion (Kuramoto--Sivashinsky) ensures that IDDA operates where nudging-based methods succeed~\cite{lunasin2017finite,markowich2016brinkman}.

After reviewing interpolated AOT and its limitations, we present IDDA's theoretical framework and implementation in Section~\ref{sec:methods}. We establish rigorous convergence theorems proving exponential error decay under explicit conditions on observation density and nudging strength. Section~\ref{sec:numerical_examples} demonstrates the effectiveness of IDDA through numerical results in benchmark cases, underscoring its potential for practical applications in fields such as weather forecasting, oceanography, and fluid dynamics. We validate IDDA through experiments on Burgers', Kuramoto--Sivashinsky, and Navier--Stokes equations. These experiments show that IDDA achieves faster and more reliable convergence than interpolated AOT in the scenarios tested, even with sparse data, consistently matching the theoretical predictions. We conclude in Section~\ref{sec:conclusion}.

\section{Continuous Data Assimilation Methods}
\label{sec:methods}

Consider the initial value problem for the PDE system
\begin{equation}
\pderiv{\Bu}{t} = \mcg[\Bu],
\label{eq:PDE_for_u_general}
\end{equation}
describing the time derivative of spatially dependent variable $\Bu(\Bx,t)$, where $\mcg[\Bu]$ is generally a nonlinear spatial differential operator. For Burgers' equation, $\mcg[\Bu]= -\Bu \grad \Bu + \paramMu \laplace \Bu$. In many situations, the \term{reference solution} $\Bu(\Bx,t)$ of~\eqref{eq:PDE_for_u_general} cannot be fully observed. For example, the initial conditions at time $t=T_0$ may be unknown, with observations of state $\Bu(\Bx,t)$ only at locations $\Bx_k$ for some period before $T_0$. Equation~\eqref{eq:PDE_for_u_general} requires initial conditions $\Bu(\Bx,T_0)$ before advancing the solution.

To determine the initial conditions, continuous data assimilation defines a modified version of \eqref{eq:PDE_for_u_general} with \term{assimilated solution} $\Bv(\Bx,t)$ that converges to the reference solution: $\Bv(\Bx,t) \rightarrow \Bu(\Bx,t)$. The error $\normLtwo{\Bv(\cdot,t)-\Bu(\cdot,t)}$ becomes small as $t$ increases. When simultaneous observations are available, the assimilated solution $\Bv(\Bx,t)$ converges to $\Bu(\Bx,t)$ for sufficiently large $t$. This approach enables the numerical approximation of $\Bu(\Bx,t)$ without requiring knowledge of the initial conditions. In practice, neither the PDE model nor the observational data is perfect. Continuous data assimilation accounts for both model and observation uncertainties. It avoids rapid, unrealistic corrections in the assimilated state while maintaining physical consistency. If historical data are available for the interval $0 \leq t \leq T_0$, the assimilated solution at $t = T_0$ can serve as the initial condition for the original PDE: $\Bu(\Bx,T_0) = \Bv(\Bx,T_0)$. The original PDE then predicts the solution for $t > T_0$.

Continuous data assimilation algorithms utilize observed data to define forcing functions derived from the interpolated discrepancy, which nudge the assimilated solution $\Bv$ towards the reference solution $\Bu$. These functions represent approximations (interpolants or projections) constructed from observational data. Projections onto low-dimensional Fourier space were widely used in many AOT-based numerical simulations. In this work, focusing on discrete measurements at locations $\Bx_k$, $\sdtd(\Bx,t)$ represents spatial interpolants constructed using the difference between observational data $\Bu(\Bx_k,t)$ and assimilated solution values $\Bv(\Bx_k,t)$ at $\Bx_k$. The interpolation operator $\widetilde{\cdot}$ depends on the observation spacing $h$ (characteristic distance between observation points) and the chosen interpolation scheme (linear, spline, etc.). Unlike projections, spatial interpolation can introduce errors that must be carefully controlled.

In this section, we denote by 
$u(\Bx,t)$ the \emph{reference solution} of~\eqref{eq:PDE_for_u_general}, 
$v(\Bx,t)$ the \emph{assimilated solution}, 
and $d = u - v$ the corresponding error between reference and assimilated solutions. 
The interpolated discrepancy $\tilde{d}$ is obtained by applying the spatial interpolation operator 
$\widetilde{\cdot}$ to the discrete set of observational differences 
$\{d(\Bx_k,t)\}$ collected at observation points $x_k$ with spatial resolution $h$. 
The analysis involves several quantities introduced in the following assumptions, including 
the nudging parameter~$\lambda$, diffusion coefficient~$\mu$, 
interpolation constant~$C$, and Lipschitz constant~$L$. 
All norms $\|\cdot\|$ and inner products $\langle\cdot,\cdot\rangle$ are taken in $L^2(\Omega)$ unless otherwise specified.

\subsection{Standing Assumptions}
\label{sec:assumptions}

Throughout this section, we make the following assumptions about the PDE operator, the interpolation scheme, and observations: Let $\Omega\subset\mathbb{R}^d$ be a bounded domain with periodic or homogeneous Dirichlet boundary conditions. Consider the dissipative PDE
$$
u_t = \mcf[u]+\mcd[u].
$$
Assume:

\begin{assumption}[Dissipation of $\mcd$]\label{assump:diffusion}

We assume that $\mcd[\cdot]$ is a linear dissipative operator satisfying
$$
\langle \mcd[\phi], \phi\rangle \le -\mu\|\nabla \phi\|^2,\qquad \mu>0.
$$
In cases where physical diffusion is negligible ($\mu \approx 0$), the artificial diffusion $\eta > 0$ is added as discussed in Remark~\ref{rem:art_diss}. This ensures well-posedness and provides regularization to balance interpolation errors.
\end{assumption}

\begin{assumption}[Lipschitz Continuity of $\mcf$]\label{assump:lipschitz}

The operator $\mcf[\cdot]$ satisfies a Lipschitz condition in $L^2(\Omega)$:
$$
\|\mcf[u] - \mcf[v]\| \leq L \|u-v\|
$$
for all sufficiently smooth functions $u, w$, where constant $L > 0$ depends on problem parameters and bounds on solution norms (e.g., $\|u\|_{L^\infty}$, $\|v\|_{L^\infty}$ and viscosity parameter for viscous Burgers' equation). For PDEs considered here, solution norms remain uniformly bounded in time due to absorbing sets or global attractors, justifying the treatment of $L$ as constant in convergence analysis.
\end{assumption}

\begin{assumption}[Interpolation Error Bound]\label{assump:interpolation}

For the chosen interpolation scheme applied to observations spatial resolution of $h$, there exists constant $C > 0$ such that for any sufficiently smooth function $f$ with $f \in H^1(\Omega)$ and bounded gradient:
$$
\|f - \tilde{f}\| \leq C h \|\nabla f\|,
$$
where $\tilde{f}$ is the interpolant constructed from samples with resolution $h$. The constant $C$ depends on the interpolation method and the geometry of the domain.
\end{assumption}

\begin{assumption}[Observation Quality]\label{assump:observations}

Observations $\{u(x_k, t)\}$ are exact measurements of the true solution at discrete points $x_k \in \Omega$ for $t \geq 0$. Extension to noisy observations is deferred to future work.
\end{assumption}

Before presenting the main convergence theorem for IDDA, we establish a key technical inequality that quantifies how well the interpolated discrepancy $\sdtd$ approximates the true error $d$ in the $L^2$ inner product. This result is fundamental to the convergence analysis.

\begin{lemma}[Interpolation Inner Product Bound]\label{lem:interp_coercivity}
Let $\,\widetilde{\cdot}: C^1(\Omega) \to C(\Omega)$ be an interpolation operator satisfying Assumption~\ref{assump:interpolation}:
$$
\|f - \tilde{f}\| \leq Ch\|\nabla f\|.
$$
Then for any $f \in C^1(\Omega)$,
\begin{equation}\label{eq:interp_inner_product}
\langle \tilde{f}, f\rangle \geq \alpha\|f\|^2 - \frac{C^2h^2}{2}\|\nabla f\|^2,
\end{equation}
where $\alpha\ge1/2$.
\end{lemma}

\begin{proof}
Starting from the interpolation bound,
$$
\|f - \tilde{f}\| \leq C h \|\nabla f\|,
$$
we have
$$
\langle \tilde{f} - f,\, \tilde{f} - f \rangle \leq C^2 h^2 \|\nabla f\|^2.
$$
Expanding the left-hand side gives
$$
\|\tilde{f}\|^2 - 2 \langle \tilde{f}, f \rangle + \|f\|^2 \leq C^2 h^2 \|\nabla f\|^2.
$$
Rearranging terms yields
$$
\langle \tilde{f}, f \rangle 
\geq \frac{1}{2} \big( \|\tilde{f}\|^2 + \|f\|^2 \big)
   - \frac{1}{2} C^2 h^2 \|\nabla f\|^2.
$$
For any non-trivial interpolation method, $\|\tilde{f}\|\ge c\|f\|$ for some $c\ge 0$ (in the general case $c=0$),
$$
\langle \tilde{f}, f \rangle 
\geq \alpha \|f\|^2 - \tfrac{1}{2} C^2 h^2 \|\nabla f\|^2,
$$
where $\alpha=\frac{1+c}{2}\ge \frac{1}{2}$, which establishes~\eqref{eq:interp_inner_product}.
\end{proof}

\subsection{Convergence Theory for Interpolated Discrepancy Data Assimilation}
\label{sec:IDDA}

We now present the main theoretical result: rigorous exponential convergence of IDDA under explicit, verifiable conditions on observation density and nudging strength.

Assume that the operator $\mcg$ splits into non-diffusive ($\mcf$) and diffusive ($\mcd$) terms: $\mcg[u]=\mcf[u]+\mcd [u]$. For instance, $\mcd[u]=\paramMu\laplace u$ with $\paramMu > 0$. The IDDA formulation modifies the assimilated system by incorporating the interpolated discrepancy $\sdtd$ directly into the non-diffusive operator:
\begin{equation}
\label{eq:IDDA_model}
  \pderiv{v}{t}=\mcf[v+\sdtd] + \mcd[v] + \paramLambda\sdtd,
\end{equation}
where $\paramLambda > 0$ is the nudging parameter that controls the target convergence rate.

\begin{remark}[Operator Classification]\label{rem:operator_classification}
The decomposition $\mcg[u] = \mcf[u] + \mcd[u]$ is not unique for all PDEs. Appropriate classification is essential for IDDA's effectiveness. The guiding principle: $\mcf$ should contain terms that drive nonlinear dynamics, instabilities, or pattern formation (where incorporating observational information through $\mcf[v+\tilde{d}]$ most benefits accuracy). In contrast, $\mcd$ contains standard stabilizing dissipation providing regularization independent of observations.
\end{remark}


\begin{theorem}[IDDA Exponential Convergence]\label{thm:IDDA_convergence}
Let $\Omega\subset\mathbb{R}^d$ be a bounded domain with periodic or homogeneous Dirichlet boundary conditions, and consider the dissipative PDE
$$   
u_t = \mcf[u]+\mcd[u],
$$
where $u(\mathbf{x},t)$ is the reference solution. Assume that Assumptions~\ref{assump:diffusion}--\ref{assump:observations} hold. Let $v(\mathbf{x},t)$ be the assimilated solution satisfying the IDDA dynamics
\begin{equation}\label{eq:IDDA_dynamics_thm}
v_t = \mcf[v+\sdtd] + \mcd[v] + \lambda\,\sdtd, 
\qquad d := u-v.
\end{equation}
If the nudging parameter $\lambda$ and observation spacing $h$ satisfy
\begin{equation}\label{eq:IDDA_cond}
\frac{L^2 C^2 h^2}{2\alpha \mu}<\lambda<\frac{\mu}{C^2h^2},
\end{equation}
then the $L^2$ error norm decays exponentially:
\begin{equation}\label{eq:IDDA_decay}
    \|d(\Bx,t)\| \le e^{-\gamma t}\,\|d(\Bx,0)\|,
    \qquad
    \gamma = \lambda \alpha - \frac{L^2 C^2 h^2}{2\mu} > 0.
\end{equation}
\end{theorem}

\begin{proof}[Proof of Theorem~\ref{thm:IDDA_convergence}]

~

Subtracting the true equation $u_t = \mcf[u] + \mcd[u]$ from the IDDA dynamics \eqref{eq:IDDA_dynamics_thm} yields the error evolution equation:
\begin{equation}\label{eq:IDDA_error_evolution}
\partial_t d = \mcd[d] + \big(\mcf[u]-\mcf[v+\sdtd]\big) - \lambda\,\sdtd.
\end{equation}
Taking the $L^2$ inner product of \eqref{eq:IDDA_error_evolution} with $d$ gives
\begin{equation}\label{eq:IDDA_energy}
\frac{1}{2}\frac{d}{dt}\|d\|^2
= \langle \mcd[d],d\rangle + \langle \mcf[u]-\mcf[v+\sdtd], d\rangle - \lambda \langle \sdtd, d\rangle.
\end{equation}
We now bound each term on the right-hand side.\\
\textbf{Dissipation term.} By Assumption~\ref{assump:diffusion},
\begin{equation}\label{eq:dissipative_bound}
\langle \mcd[d], d\rangle \le -\mu\|\nabla d\|^2.
\end{equation}
\textbf{Non-diffusive term.} Using the Lipschitz condition from Assumption~\ref{assump:lipschitz}:
$$
\|\mcf[u]-\mcf[v+\sdtd]\|
\le L\|u-(v+\sdtd)\|
= L\|d-\sdtd\|.
$$
Applying the interpolation error bound from Assumption~\ref{assump:interpolation}:
$$
\|d-\sdtd\| \le C h \|\nabla d\|.
$$
Combining these with the Cauchy-Schwarz inequality:
\begin{equation}\label{eq:nonlinear_bound}
\langle \mcf[u]-\mcf[v+\sdtd], d\rangle
\le \|\mcf[u]-\mcf[v+\sdtd]\|\|d\|
\le L C h \|\nabla d\|\|d\|.
\end{equation}
\textbf{Nudging term.} Applying Lemma~\ref{lem:interp_coercivity} with $\alpha = 1/2$:
\begin{equation}\label{eq:nudging_bound}
-\,\lambda\langle \sdtd, d\rangle
\le -\,\lambda\alpha\|d\|^2 + \frac{1}{2}\lambda C^2h^2 \|\nabla d\|^2.
\end{equation}
Substituting bounds \eqref{eq:dissipative_bound}, \eqref{eq:nonlinear_bound}, and \eqref{eq:nudging_bound} into \eqref{eq:IDDA_energy}:
\begin{equation}\label{eq:combined_before_young}
\frac{1}{2}\frac{d}{dt}\|d\|^2
\le -\mu\|\nabla d\|^2
+ L C h \|\nabla d\|\|d\|
- \frac{\lambda}{2}\|d\|^2
+ \frac{\lambda C^2h^2}{2} \|\nabla d\|^2.
\end{equation}
To handle the mixed term $LCh\|\nabla d\|\|d\|$, we apply Young's inequality. We use the form\footnote{Young's inequality has multiple equivalent forms: $ab \leq \frac{a^2}{2\epsilon} + \frac{\epsilon b^2}{2}$ or $ab \leq \frac{\epsilon a^2}{2} + \frac{b^2}{2\epsilon}$. The choice of form determines which term involves $\epsilon$ in the numerator versus denominator. We use the second form here.}
\begin{equation}\label{eq:young_form}
ab \leq \frac{\epsilon a^2}{2} + \frac{b^2}{2\epsilon}
\end{equation}
with $a = LCh\|\nabla d\|$, $b = \|d\|$, and parameter choice $\epsilon = \frac{\mu}{L^2C^2h^2}$. This yields:
\begin{equation}\label{eq:mixed_term}
L C h \|\nabla d\|\|d\|
\le \frac{\mu}{2}\|\nabla d\|^2 + \frac{L^2 C^2 h^2}{2\mu}\|d\|^2.
\end{equation}
Substituting \eqref{eq:mixed_term} into \eqref{eq:combined_before_young}:
\begin{align}
&\frac{1}{2}\frac{d}{dt}\|d\|^2
\le - \Big(\lambda\alpha-\frac{L^2 C^2 h^2}{2\mu}\Big)\|d\|^2
- \frac{1}{2}\Big(\mu-\lambda C^2 h^2\Big)\|\nabla d\|^2.
\label{eq:final_energy_bound}
\end{align}

Under condition \eqref{eq:IDDA_cond}, both coefficients are strictly positive. Discarding the non-negative gradient term in \eqref{eq:final_energy_bound}:
\begin{equation}\label{eq:gronwall_ready}
\frac{d}{dt}\|d\|_{L^2}^2 \le -2\gamma \|d\|_{L^2}^2,\qquad
\gamma = \lambda\alpha - \frac{L^2 C^2 h^2}{2\mu} > 0.
\end{equation}
Applying Gr\"onwall's inequality to \eqref{eq:gronwall_ready}:
$$
\|d(t)\|^2 \le e^{-2\gamma t}\|d(0)\|^2,
$$
which immediately gives \eqref{eq:IDDA_decay} by taking square roots.
\end{proof}

\subsection{Comparison and Interpretation}
\label{sec:comparison}

Building upon the exponential convergence established in Theorem~\ref{thm:IDDA_convergence}, 
we now examine its analytical implications, compare the proposed IDDA formulation with the standard interpolated AOT method, 
and discuss the practical considerations relevant to implementation. 
This subsection unifies the interpretation of the theoretical bounds, parameter dependencies, 
and numerical behavior observed in Section~\ref{sec:numerical_examples}.

\paragraph{Comparison with interpolated AOT.}
The standard interpolated AOT approach augments the model dynamics with a nudging term proportional to the interpolated discrepancy:
\begin{equation*}
\frac{\partial v}{\partial t}
   = \mcf[v] + \mcd[v] + \lambda\,\sdtd.
\end{equation*}
Under identical assumptions, AOT satisfies the convergence condition 
(see Appendix~\ref{sec:AOT_appendix} for detailed proof)
\begin{equation}\label{eq:AOT_cond_comparison}
\frac{L}{\alpha} < \lambda_{\mathrm{AOT}} < \frac{2\mu}{C^2h^2},
\end{equation}
with decay rate $\gamma_{\mathrm{AOT}} = \lambda\alpha - L$ for the $L^2$ norm. 
The lower bound comparison
$$
\lambda_{\mathrm{AOT}} > \frac{L}{\alpha},
\qquad 
\lambda_{\mathrm{IDDA}} > \frac{L^2C^2h^2}{2\alpha\mu}
$$
highlights the difference in sensitivity to the Lipschitz constant~$L$: 
for reasonably dense observations (small~$h$), IDDA's lower bound becomes negligible while AOT's remains fixed at~$O(L)$. 
This explains IDDA's superior performance with moderate nudging parameters: 
the structural modification $\mcf[v] \rightarrow \mcf[v+\sdtd]$ removes the direct dependence on~$L$ from the convergence condition, 
replacing it with a term that vanishes as~$O(h^2)$. 

For a given~$\lambda$, the expected convergence rates satisfy
$$
\gamma_{\mathrm{AOT}} \approx \lambda\alpha - L,
\qquad 
\gamma_{\mathrm{IDDA}} \approx \lambda\alpha - O(h^2),
$$
so the convergence rate of AOT may be reduced due to rapid change in the dynamics of PDEs.
In the fully observed or fine-resolution limit ($h \to 0$, $\alpha \to 1$), 
the decay rate in~\eqref{eq:IDDA_decay} approaches $\lambda$ for $\|d\|_{L^2}$. 
Consequently, IDDA achieves an exponential convergence rate that is closely related to the user-defined feedback strength $\lambda$, 
a trend consistently observed in numerical experiments even under relatively sparse observations. 

It appears that AOT admits a formally large upper bound on~$\lambda$. 
However, for dissipative systems with sufficiently fine observational resolution, 
a stronger observability condition on the interpolation operator typically holds:
$$
\langle \tilde{f}, f \rangle \ge \alpha^* \|f\|_{L^2}^2,
$$
for some constant $\alpha^* \in (0,1]$. 
Here, $\alpha^*$ quantifies the fraction of the system’s energy represented within the observed (interpolated) subspace. 
Under this condition, the $L^2$ energy estimates for IDDA and AOT become
$$
\frac{1}{2}\frac{d}{dt}\|d\|_{L^2}^2
\le -\Big(\lambda_{\mathrm{IDDA}}\alpha^* - \frac{L^2 C^2 h^2}{2\mu}\Big)\|d\|_{L^2}^2
- \frac{1}{2}\mu\|\nabla d\|_{L^2}^2,
$$
and
$$
\frac{1}{2}\frac{d}{dt}\|d\|_{L^2}^2
\le -\Big(\lambda_{\mathrm{AOT}}\alpha^* - L\Big)\|d\|_{L^2}^2
- \frac{1}{2}\mu\|\nabla d\|_{L^2}^2,
$$
respectively. 
Consequently, for sufficiently small observation spacing~$h$, 
an effective nudging parameter~$\lambda$ always exists and, in practice, 
does not require an explicit upper bound. 
This reflects the empirical observation that both AOT and IDDA remain stable for large~$\lambda$ 
once the interpolation captures all dynamically determining scales. If important active modes are omitted ($h$ is not sufficiently small and $\alpha^* \approx 0$) 
and the physical diffusion~$\mu$ is weak, 
the feedback term cannot enforce convergence regardless of the value of~$\lambda$. 
This condition formalizes the intuitive requirement that the observational network must resolve all dynamically determining scales of the system.

\paragraph{Evaluation of nonlinear terms in practice.}
The Lipschitz condition for the operator $\mcf[u]$ is imposed with respect to the state variable~$u$. 
In numerical implementations, the modified term $\mcf[v+\sdtd]$ should be evaluated in a manner consistent with the smoothness of the interpolation method. 
For instance, in the viscous Burgers equation where $\mcf[u] = u\,u_x$, 
if $\sdtd$ is obtained via piecewise linear interpolation, 
it is preferable to compute
$$
\mcf[v+\sdtd] = (v+\sdtd)\,v_x,
$$
rather than directly differentiating $\sdtd$ as in $(v+\sdtd)(v+\sdtd)_x$, 
to avoid artificial oscillations or singularities caused by the nonsmooth derivative~$\sdtd_x$. 
Such implementation choices ensure that the analytical assumptions underlying the convergence theorem remain valid at the discrete level.

\begin{remark}[Artificial Dissipation] \label{rem:art_diss}
If physical diffusion is negligible ($\mu \approx 0$), we can optionally add artificial dissipation to the interpolated discrepancy:
\begin{equation}
\label{eq:IDDA_with_eta}
  \pderiv{v}{t}=\mcf[v+\sdtd] + \mcd[v] + \paramLambda\sdtd - \paramEta\laplace\sdtd, \quad (\paramEta > 0).
\end{equation}
This provides additional regularization at the interpolation scale without over smoothing the physical dynamics. The convergence result extends directly if the interpolation operator $\widetilde{\cdot}$ satisfies the following $H^1$-stability and alignment condition:  
there exists a constant $\kappa\in(0,1]$ such that
\begin{equation}\label{assump:H1-coercive}
\|\nabla \tilde{f}\| \le \|\nabla f\|,
\qquad 
\langle \nabla \tilde{f}, \nabla f\rangle \ge \kappa \|\nabla \tilde{f}\|^2,
\quad \forall f \in H^1(\Omega),
\end{equation}
which implies
$$
\eta\langle \Delta\tilde{d}, d\rangle
= -\,\eta\langle \nabla\tilde{d}, \nabla d\rangle
\le -\,\eta\kappa\,\|\nabla\tilde{d}\|^2
\le -\,\eta\kappa\,\|\nabla d\|^2.
$$
Then
$$
\frac{1}{2}\frac{d}{dt}\|d\|^2
\le -\Big(\lambda\alpha - \frac{L^2 C^2 h^2}{2(\mu+\eta\kappa)}\Big)\|d\|^2
- \frac{1}{2}\Big((\mu+\eta\kappa) - \lambda C^2 h^2\Big)\|\nabla d\|^2.
$$
If the parameters $\lambda$ and $h$ satisfy
\begin{equation}\label{eq:IDDA_cond_eta}
\frac{L^2 C^2 h^2}{2\alpha(\mu+\eta\kappa)} 
\;<\; \lambda \;<\; \frac{\mu+\eta\kappa}{C^2 h^2},
\end{equation}
then
$$
\|d(t)\| \le e^{-\gamma_\eta t}\|d(0)\|,
\qquad
\gamma_\eta = \lambda\alpha - \frac{L^2 C^2 h^2}{2(\mu+\eta\kappa)} > 0.
$$
The coefficient $\kappa \in (0,1]$ in~\eqref{assump:H1-coercive} quantifies how well the interpolation 
$\widetilde{\cdot}$ preserves energy alignment in the $H^1$ inner product.
It measures the degree to which the interpolated gradient $\nabla\tilde{f}$ aligns with the true gradient $\nabla f$:
$$
\kappa = \inf_{f\neq 0} 
\frac{\langle \nabla \tilde{f}, \nabla f\rangle}{\|\nabla \tilde{f}\|^2}.
$$
For exact $H^1$ projections (e.g., Galerkin projection onto a finite element or Fourier subspace), $\kappa=1$. For linear or cubic spline interpolation on uniform grids, $\kappa$ typically satisfies $\kappa=1-O(h^2)$ depending on boundary conditions and grid resolution (see \cite{brenner2008mathematical}). For smooth kernel-based interpolations (e.g., Gaussian or RBF), $\kappa$ decreases slightly with kernel width $h$, but remains positive for stable interpolation schemes.
The constant $\kappa$ thus represents the ``energy efficiency'' of the interpolation: larger $\kappa$ improves dissipation alignment and expands the admissible range of $\lambda$ in~\eqref{eq:IDDA_cond_eta}. 
In the limit $\kappa\to1$, the artificial dissipation term contributes its full stabilizing strength $\eta$ to the effective diffusion $\mu+\eta$.

\end{remark}

\begin{remark}[Time-Dependent Nudging Parameter] \label{rem:time_dep_rate}
Choosing $\lambda=\lambda(t)$ allows time-dependent convergence rates, potentially useful for adaptive assimilation strategies (\cite{larios2024nonlinear}). However, rapidly varying $\lambda(t)$ may require smaller time steps to maintain numerical stability.
\end{remark}

\begin{remark}[Time-Discrete Observations] \label{rem:time_disc_obs}
If observations $\{u(x_k,t_n)\}$ are discrete in time as well as space, spatio-temporal interpolation can be used to construct $\sdtd(\mathbf{x},t)$ for all $t$. The convergence analysis requires additional care to account for temporal interpolation errors, but the fundamental mechanism remains unchanged.
\end{remark}

\paragraph{Conceptual summary.}
Overall, IDDA utilizes observational information in two complementary ways: 
the additive term $\lambda\sdtd$ enforces direct relaxation toward the observed state, 
while the modified nonlinearity $\mcf[v+\sdtd]$ implicitly corrects the model dynamics. 
This dual mechanism enhances both stability and convergence speed. IDDA is more robust since its nonlinear correction mitigates part of the interpolation error, and achieves faster and more reliable convergence across a range of dissipative PDEs, as confirmed in the numerical experiments presented in Section~\ref{sec:numerical_examples}.

\section{Numerical Examples}
\label{sec:numerical_examples}

We demonstrate IDDA versus interpolated AOT with sparse data across several benchmark PDEs.

\subsection{Numerical Implementation Details}
\label{sec:interp_implementation}

Spatial interpolation of observational data and model fields is performed using standard MATLAB routines, employing piecewise linear or cubic spline interpolation in one dimension and radial basis function (RBF) interpolation in two dimensions, depending on the problem setup.

The governing PDEs are discretized in space using finite differences on uniform grids, with periodic boundary conditions. Time integration is carried out using MATLAB's \texttt{ode45} solver, which implements an explicit adaptive Runge--Kutta (4,5) method. The adaptive time stepping automatically satisfies stability and accuracy requirements for the range of parameters and diffusion coefficients examined in the numerical experiments.

Convergence rates are evaluated from the time evolution of the $L^2$ error 
$E(t) = \|u(\cdot,t) - v(\cdot,t)\|$. 
Assuming exponential decay $E(t) \approx E_0 e^{-\gamma t}$ predicted by Theorem~\ref{thm:IDDA_convergence}, 
the rate $\gamma$ is estimated from the slope of $\log E(t)$ over the interval exhibiting clear exponential behavior. 
Linear regression of $\log E(t)$ versus time yields the fitted rate $\rho \approx \gamma$, 
which is compared directly with the theoretical prediction $\gamma \approx \lambda\alpha$. 
All reported values exclude initial transients and late-time plateaus near machine precision.
\begin{table}[htbp]
\centering
\caption{Summary of numerical experiments demonstrating IDDA performance across representative dissipative PDEs. 
All simulations employ Runge--Kutta (4,5) time integration. 
Convergence plots display $\log(\|d\|)$ versus time, from which the exponential decay rate $\gamma$ is extracted.}
\label{tab:experiments_summary}
\begin{small}
\begin{tabular}{@{}lcccccl@{}}
\toprule
\textbf{Equation} & \textbf{Dim} & $\boldsymbol{N_s}$ & $\boldsymbol{N_x}$ & $\boldsymbol{\lambda}$ & $\boldsymbol{\mu}$ & \textbf{Interpolation} \\
\midrule
Viscous Burgers' & 1D & 3--100 & 1000 & 2 & 0.001 & Linear \\
KPP-Burgers' & 1D & 3 & 1000 & 1--200 & 0.01 & Linear/Cubic spline \\
Kuramoto-Sivashinsky & 1D & 8-80 & 1024 & 2 & N/A & Cubic spline \\
Navier-Stokes & 2D & 100--1000 & $256^2$ & 2 & $10^{-4}$ & $C^2$ RBF \\
\bottomrule
\end{tabular}
\end{small}
\end{table}

Further implementation details for each test equation are provided in the corresponding subsections.

\subsection{Viscous Burgers' Equation}
\label{sec:Burgers}

We consider the viscous Burgers' equation
$$
\pderiv{u}{t} = -u\,u_x + \mu\,u_{xx},
$$
and compare the AOT and IDDA formulations:
\begin{align}
\label{eq:Burgers_AOT_labelled}
&\text{AOT:} \quad \pderiv{v}{t} = -v\,v_x + \mu\,v_{xx} + \lambda\,\sdtd, \\
\label{eq:Burgers_IDDA_labelled}
&\text{IDDA:} \quad \pderiv{v}{t} = -(v+\sdtd)\,v_x + \mu\,v_{xx} + \lambda\,\sdtd.
\end{align}
The computational domain is $x \in [0,1]$ with periodic boundary conditions and diffusion coefficient $\mu = 0.001$.  
The reference initial condition
$$
u(x,0) = 1 + \sin(2\pi x) + \cos^2(4\pi x).
$$
Spatial discretization uses $N_x = 1000$ uniform grid points, and observations are available at 
$N_s = 3$ points $\{0.16, 0.49, 0.82\}$ with linear interpolation ($h \approx 0.33$).  
The assimilated initial condition is $v(x,0) = 0$.

Figure~\ref{fig:Burgers} compares the performance of IDDA and interpolated AOT for the viscous Burgers' equation.  
Panel~(a) shows results using three observation points ($N_s = 3$).  
IDDA rapidly aligns with the reference solution, accurately reproducing both amplitude and phase of the evolving field, while AOT lags behind and exhibits visible phase shifts.  
The $L^2$ error curves confirm exponential convergence: IDDA achieves a rate $\gamma \approx 2.02$, closely matching the target $\lambda = 2$, whereas AOT converges more slowly at $\gamma \approx 0.88$.  

Panel~(b) examines how the convergence rate varies with the number of observation points $N_s$ (from 3 to 100).  
IDDA consistently reaches the target rate $\lambda = 2$ even with very sparse data ($N_s = 3$), demonstrating robustness to limited observations.  
In contrast, AOT improves only gradually with increasing $N_s$ and still fails to achieve the target rate even at $N_s = 100$.  
These results confirm that incorporating the interpolated discrepancy directly within the nonlinear operator enables IDDA to assimilate key flow structures more efficiently than standard nudging.

\begin{figure}[H]
\begin{subfigure}{\textwidth}
\centering
\includegraphics[width=\textwidth]{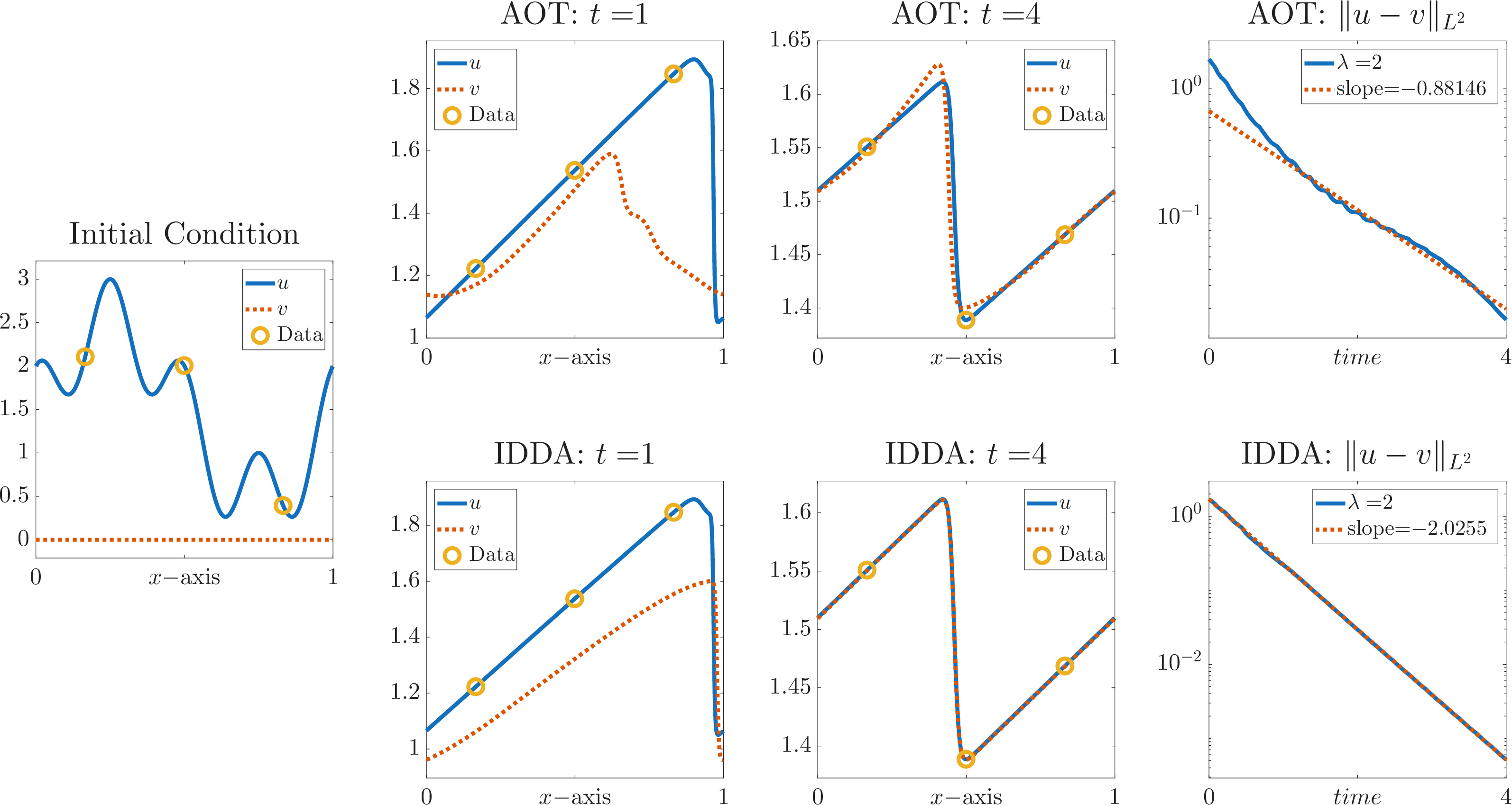}
\caption{Data assimilation for viscous Burgers' equation with $\mu=0.001$, $\lambda=2$, and $N_s=3$ observation points using linear interpolation. The left three panels show spatial profiles at $t=0, 1, 4$ of reference solution $u$ (blue solid line), assimilated solution $v$ (orange dashed line), and observations (orange dots). The right panel displays the time evolution of $L^2$ errors on a semi-logarithmic scale. Interpolated AOT results in a convergence rate $\approx 0.88$, slower than the IDDA $\approx 2.02$, which matches the target $\lambda=2$.}
\label{Fig_Burgers}
\end{subfigure}
\vspace{1em}
\begin{subfigure}{\textwidth}
\centering
\includegraphics[width=0.6\textwidth]{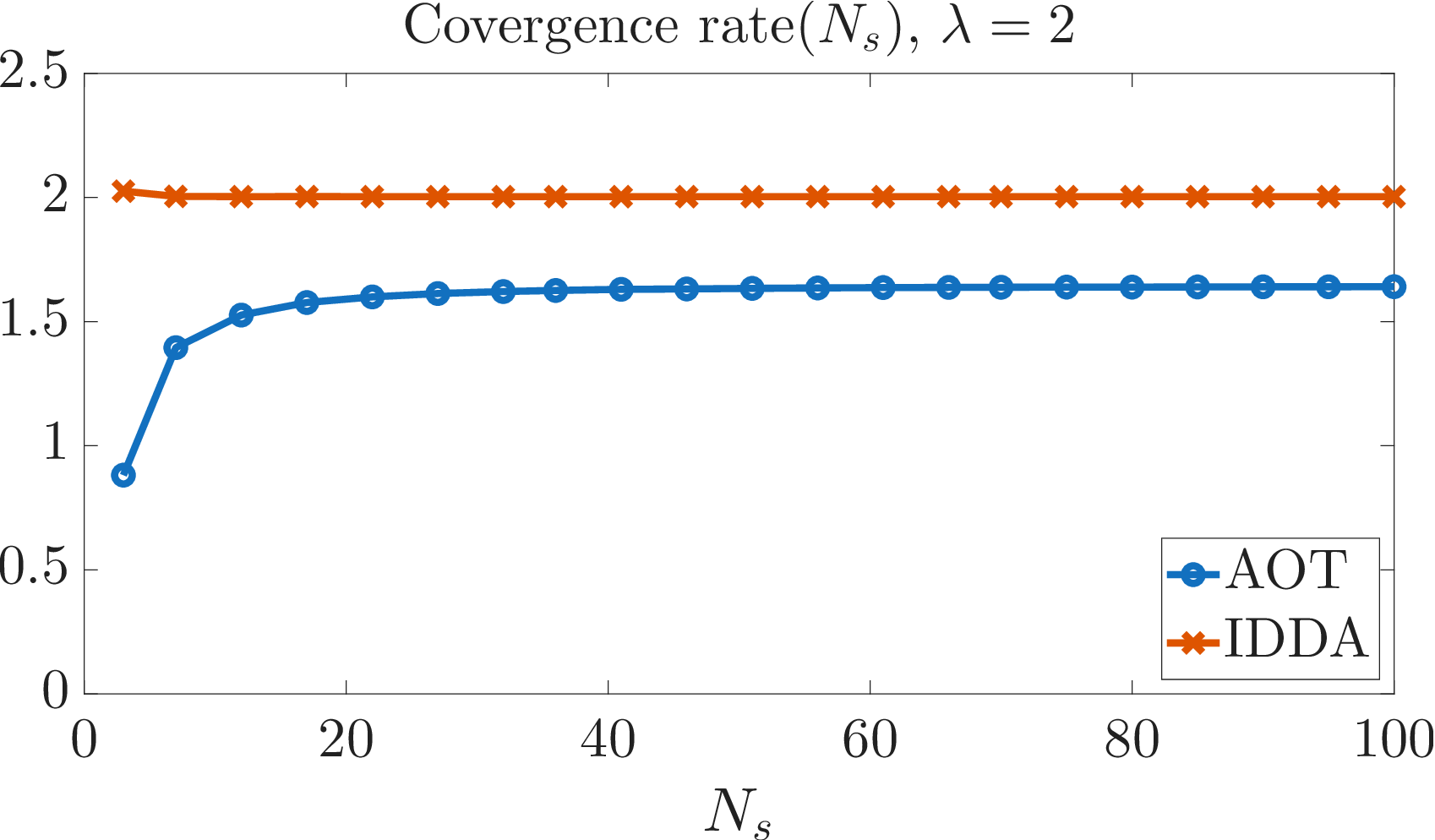}
\caption{Convergence rates versus the number of observed data points $N_s$. IDDA attains the target rate $\lambda=2$ even with few observations $N_s=3$, whereas AOT remains slower and fails to reach it even with $N_s=100$.}
\label{Fig_Burgers_Ns_Comp}
\end{subfigure}
\caption{Comparison for viscous Burgers' equation.}
\label{fig:Burgers}
\end{figure}

\subsection{KPP--Burgers' Equation}
\label{sec:KPP}
We consider the KPP--Burgers' equation
$$
\pderiv{u}{t} = -u\,u_x - 10u(u-1)(u-2) + \mu\,u_{xx},
$$
which combines advection, reaction, and diffusion and exhibits bistable dynamics with stable steady states at $u=0$ and $u=2$. We compare the AOT formulation with two IDDA variants that differ by their interpolation methods:
\begin{small}
\begin{align}
\label{eq:KPP_AOT_labelled}
&\text{AOT:} \quad 
\pderiv{v}{t} = -v\,v_x - 10v(v-1)(v-2) + \mu\,v_{xx} + \lambda\,\sdtd, \\[0.5em]
\label{eq:KPP_IDDA_spline_labelled}
&\text{IDDA--spline:} \quad 
\pderiv{v}{t} = -(v+\sdtd)(v+\sdtd)_x 
- 10(v+\sdtd)\bigl((v+\sdtd)-1\bigr)\bigl((v+\sdtd)-2\bigr)
+ \mu\,v_{xx} + \lambda\,\sdtd, \\[0.5em]
\label{eq:KPP_IDDA_linear_labelled}
&\text{IDDA--linear:} \quad 
\pderiv{v}{t} = -(v+\sdtd)\,v_x 
- 10(v+\sdtd)\bigl((v+\sdtd)-1\bigr)\bigl((v+\sdtd)-2\bigr)
+ \mu\,v_{xx} + \lambda\,\sdtd.
\end{align}
\end{small}

The computational domain is $x \in [0,1]$ with periodic boundary conditions and diffusion coefficient $\mu = 0.01$. 
The reference initial condition is $u(x,0) = 1 + \sin(2\pi x)$, and the assimilated initial condition is $v(x,0) = 0$. 
Spatial discretization uses $N_x = 1000$ uniform grid points, and observations are available at 
$N_s = 3$ points $\{0.16, 0.49, 0.82\}$ with ($h \approx 0.33$). Two interpolation methods are tested: cubic spline interpolation and piecewise linear interpolation for both AOT and IDDA. 

The KPP--Burgers equation provides a challenging test case for data assimilation because the traveling front couples nonlinear advection and reaction terms, producing both steep gradients and bistable dynamics. As shown in Figure ~\ref{Fig_KPP}, IDDA maintains smooth exponential convergence and accurately tracks the front position and amplitude, even with only three observation points. In contrast, the interpolated AOT scheme exhibits irregular error decay and incomplete synchronization of the discontinuity, indicating that its additive nudging term alone cannot fully capture the coupled nonlinear dynamics.

Figure \ref{Fig_lambda_comp} further illustrates how the convergence rate varies with the nudging parameter~$\lambda$. Both AOT and IDDA reach an upper performance limit near $8$, consistent with the physical constraint that assimilation is most effective when the traveling front passes through an observation location. For small $\lambda$, AOT fails to converge and, when $\lambda$ increases, AOT displays non-monotonic behavior due to competition between the nudging term and the intrinsic wave motion. IDDA, by incorporating the interpolated discrepancy into the nonlinear operator, achieves higher convergence for weak nudging and preserves stable behavior across a wide range of~$\lambda$.

\begin{figure}[H]
\begin{subfigure}{\textwidth}
\centering
\includegraphics[width=0.95\textwidth]{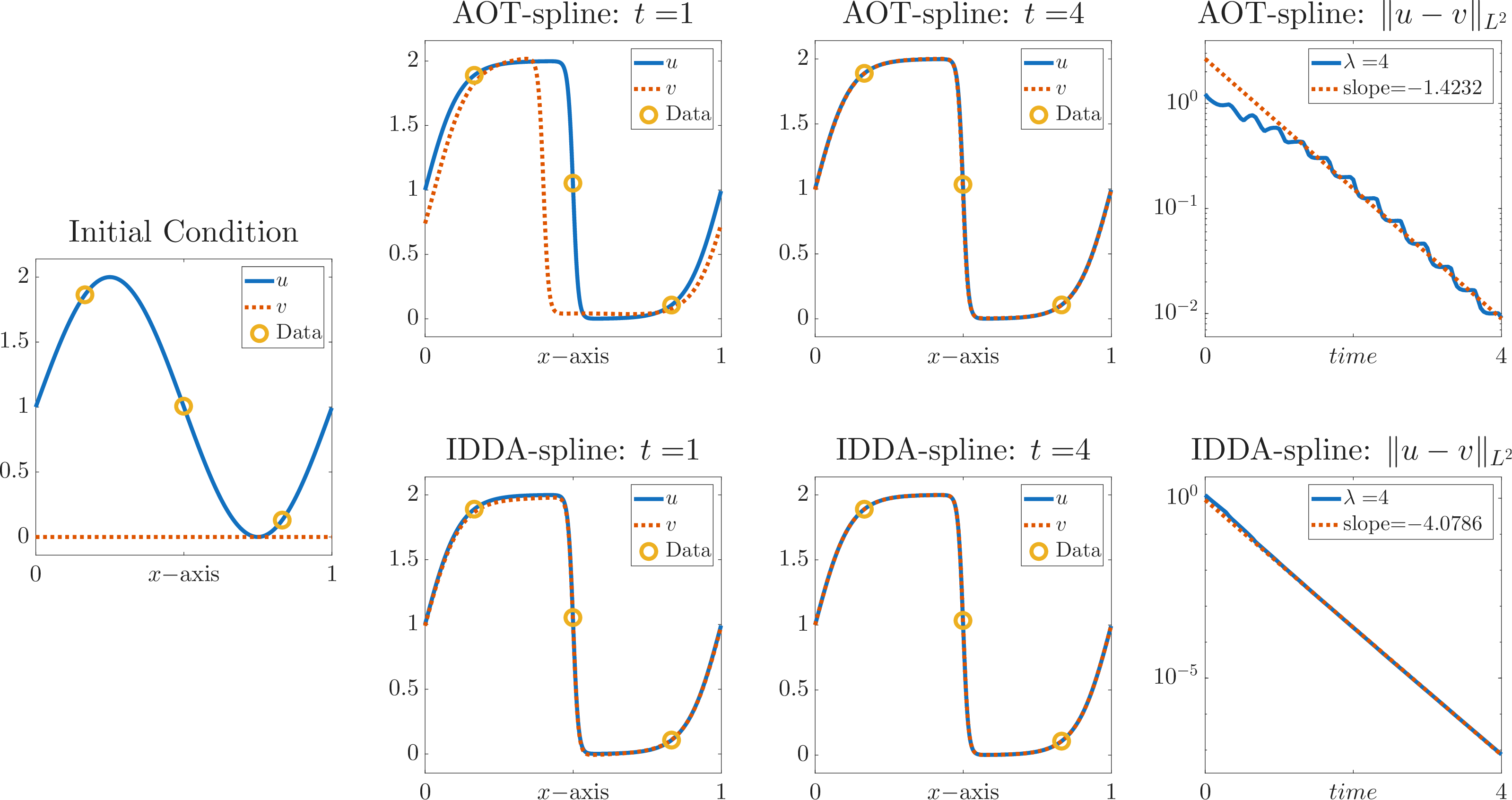}
\caption{
Data assimilation for the KPP--Burgers equation with $\mu=0.01$, $\lambda=4$, and $N_s=3$ using cubic-spline interpolation. 
Interpolated AOT shows non-monotonic convergence with final rate $\approx1.42$ and incomplete synchronization of the discontinuity. 
IDDA achieves smooth exponential convergence at $\approx4.07$, matching the theoretical $\lambda=4$. 
By correcting both advection and reaction terms through $\mcf [v+\tilde d]$, IDDA reduces the phase error in AOT and reliably tracks the coupled nonlinear dynamics.
}
\label{Fig_KPP}
\end{subfigure}

\vspace{1em}

\begin{subfigure}{\textwidth}
\centering
\includegraphics[width=0.95\textwidth]{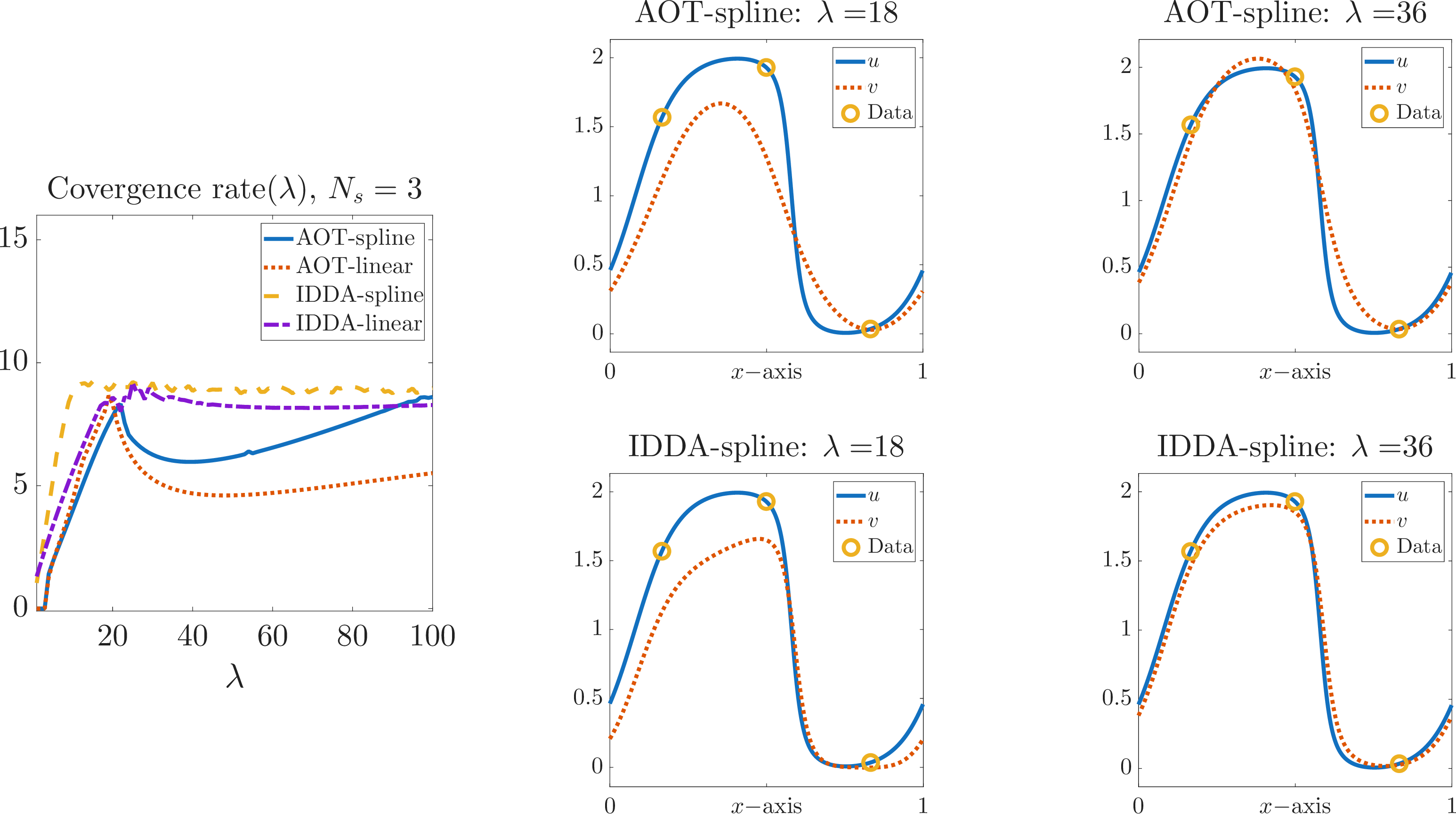}
\caption{
Convergence rate versus nudging parameter $\lambda$ for $N_s=3$. 
Both AOT and IDDA reach an upper bound near~8, reflecting that effective nudging occurs only when the traveling front passes observation points. 
AOT fails for small~$\lambda$ and shows non-monotonic behavior due to interaction between the nudging term and the moving front. 
For $\lambda=18$ and~36, AOT follows the interpolated nudging term, disrupting bistable behavior and missing the true front, 
whereas IDDA maintains higher convergence, correct phase, and amplitude accuracy.
}
\label{Fig_lambda_comp}
\end{subfigure}

\caption{Comparison for the KPP--Burgers equation.}
\label{fig:KPP}
\end{figure}

\subsection{Kuramoto--Sivashinsky (KS) Equation}
\label{sec:KS}

We consider the KS equation
$$
\pderiv{u}{t} = -u\,u_x - 2u_{xx} - u_{xxxx},
$$
which exhibits spatio-temporal chaos with sensitive dependence on initial conditions and positive Lyapunov exponents. Small discrepancies in the initial or observed data can grow exponentially, making data assimilation particularly challenging. 

In this problem, the second-order term $-2u_{xx}$ acts as an anti-diffusion that destabilizes long wavelengths and initiates chaotic motion, while the fourth-order term $-u_{xxxx}$ provides short-wavelength stabilization. Within the IDDA framework, we classify $-uu_x-2u_{xx}$ as the non-diffusive operator $\mcf$ and $-u_{xxxx}$ as the dissipative operator $\mcd$. Together with cubic spline interpolation, the corresponding formulations are
\begin{align}
\label{eq:KS_AOT_labelled}
&\text{AOT:} \quad 
\pderiv{v}{t} = -v\,v_x - 2v_{xx} - v_{xxxx} + \lambda\,\sdtd, \\[0.5em]
\label{eq:KS_IDDA_labelled}
&\text{IDDA:} \quad 
\pderiv{v}{t} = -(v+\sdtd)(v+\sdtd)_x - 2(v+\sdtd)_{xx} - v_{xxxx} + \lambda\,\sdtd.
\end{align}

The computational domain is $x \in [0,32\pi]$ with periodic boundary conditions. 
The reference initial condition is $u(x,0)=\cos(x/16)(1+\sin(x/16))$, and the assimilated initial condition is $v(x,0)=0$. 
Spatial discretization uses $N_x = 1024$ uniform grid points, and observations are available at 
$N_s = 64$ uniformly distributed points with cubic spline interpolation. 
The nudging parameter is $\lambda = 2$, and diffusion coefficients are as given in the governing equation. 

The KS equation provides a severe test for data assimilation because its chaotic attractor rapidly amplifies small-scale perturbations. 
As shown in Figure~\ref{Fig:KS_comp}, IDDA achieves smooth exponential convergence with a rate $\approx 2 =\lambda$, closely matching the theoretical prediction and maintaining consistent synchronization despite chaotic forcing. 
In contrast, the interpolated AOT scheme exhibits slower convergence ($\approx 1.26$).

The KS equation represents a strongly chaotic system in which small perturbations grow rapidly due to anti-diffusive instabilities. As shown in Figure~\ref{Fig_KS}, the interpolated AOT scheme achieves a convergence rate of approximately $1.26$, significantly below the target $\lambda = 2$. In contrast, IDDA maintains stable synchronization and achieves a convergence rate of about~$2$, matching the theoretical prediction. The inclusion of $\sdtd$ within the nonlinear term $\mcf[v+\sdtd]$ allows IDDA to directly correct anti-diffusive errors and preserve the phase of chaotic fluctuations.

Figure~\ref{Fig_ns_comp} examines the dependence on the number of observation points $N_s$. When $N_s$ is small, both AOT and IDDA fail to converge, as the unresolved small scales lead to persistent phase errors and $\tilde{d}$ deviates significantly from the true discrepancy $d$. Once the number of observations exceeds twice the number of active unstable modes ($N_s \gtrsim 44$), both methods recover exponential convergence, and IDDA reaches the target rate $\lambda$ more sharply. For $N_s = 48$, the discrepancy and its interpolant nearly machine error, with IDDA showing smaller residual oscillations. These results confirm that IDDA retains accuracy and stability even in chaotic regimes, provided that the observation density is sufficient to capture the dynamically relevant modes.

\begin{figure}[H]
\begin{subfigure}{\textwidth}
\centering
\includegraphics[width=0.95\textwidth]{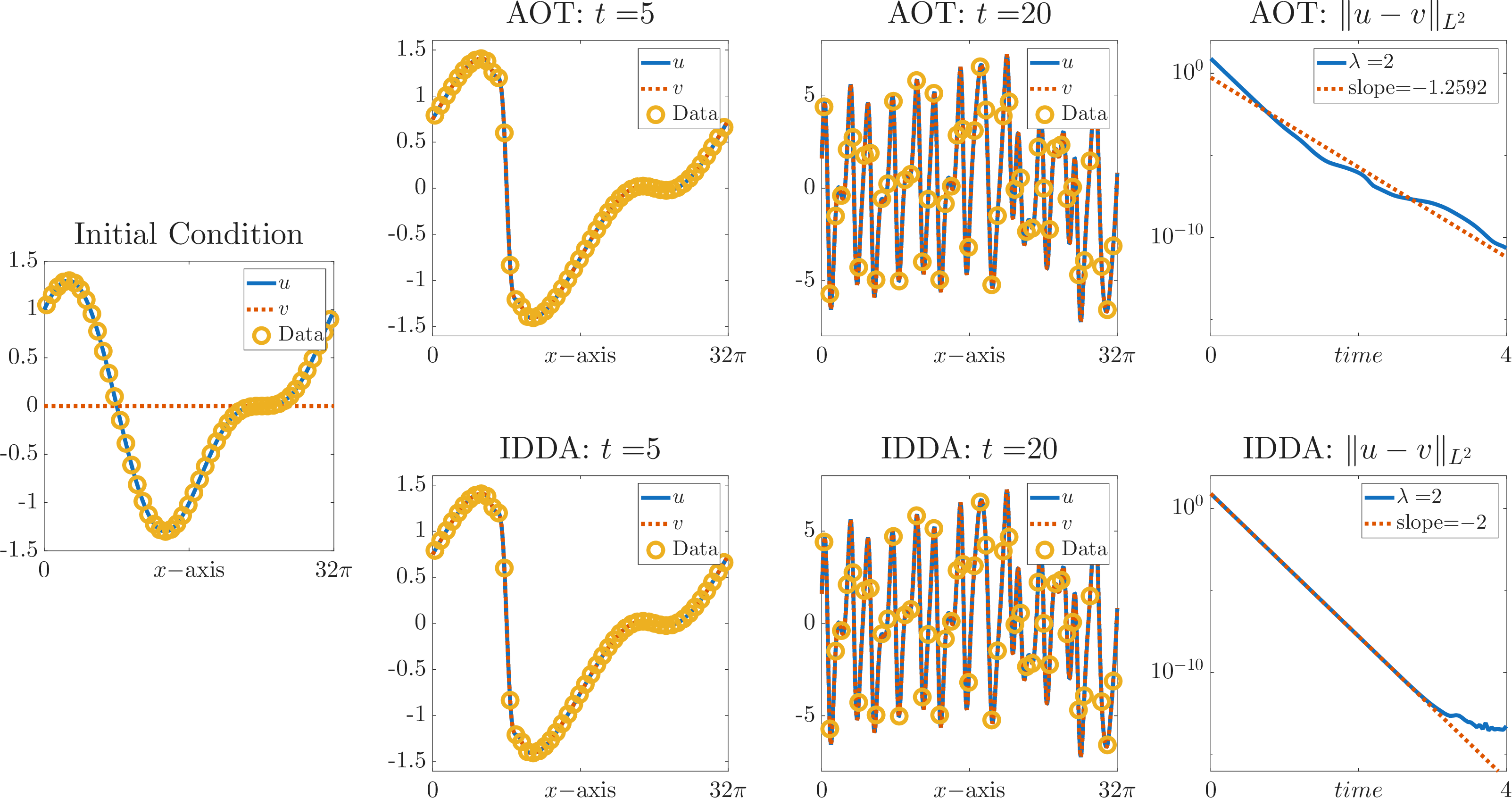}
\caption{Interpolated AOT achieves rate $\approx 1.13$, substantially below target $\lambda=2$, indicating that additive nudging alone cannot effectively counteract anti-diffusive instabilities and chaotic dynamics. IDDA achieves rate $\approx 2$, matches the target $\lambda=2$.}
\label{Fig_KS}
\end{subfigure}
\vspace{1em}
\begin{subfigure}{\textwidth}
\centering
\includegraphics[width=0.95\textwidth]{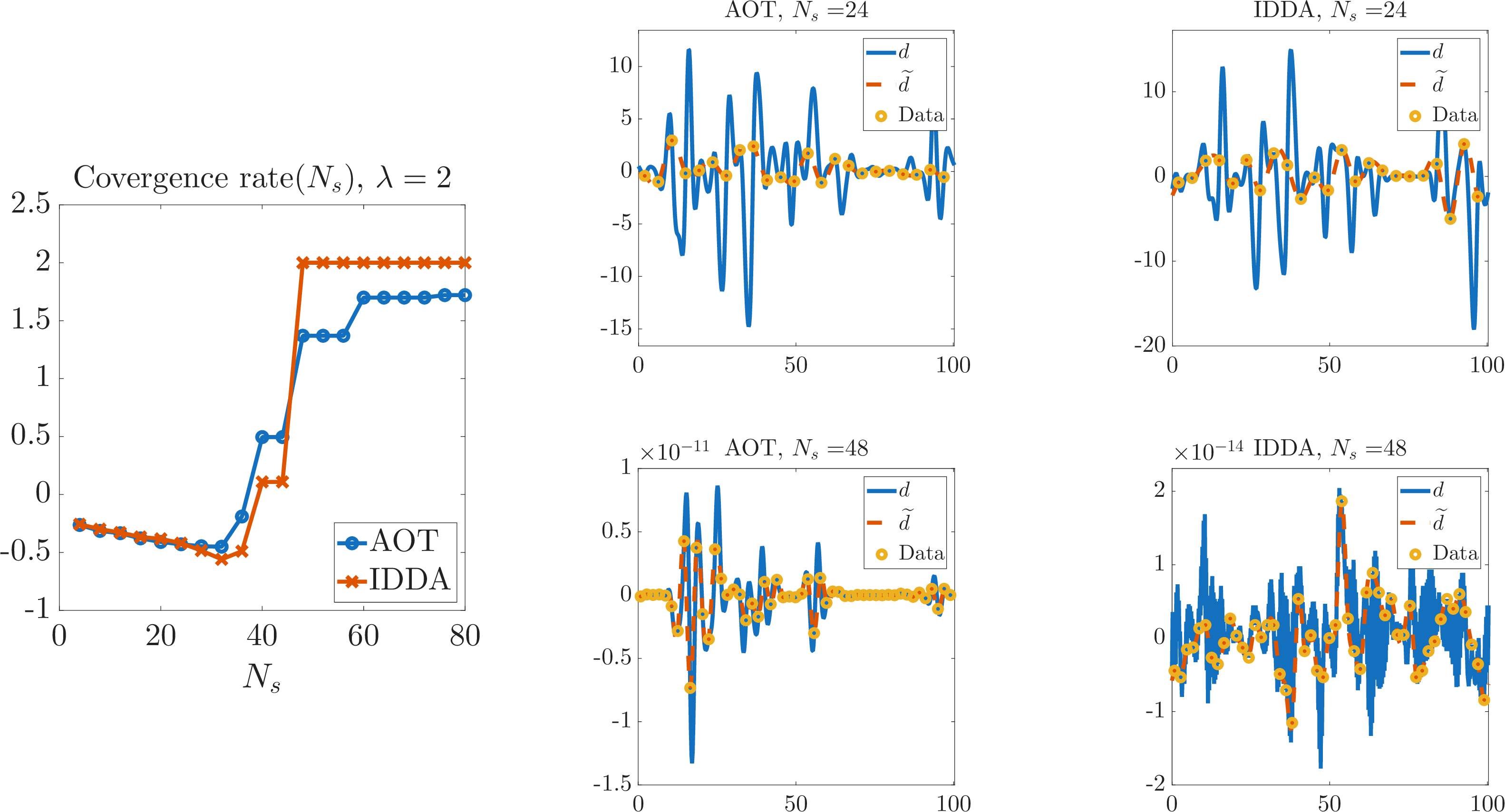}
\caption{Convergence behavior with respect to the number of observation points $N_s$ 
for the 1D Kuramoto--Sivashinsky-type equation.
The left panel shows the convergence rate versus $N_s$ for AOT and IDDA with $\lambda = 2$. 
Both methods fail when the number of observations is insufficient to resolve all active modes, resulting in negative or near-zero rates. Once $N_s$ exceeds twice the number of active modes, both schemes recover exponential convergence, 
with IDDA achieving the target rate $\lambda$ more sharply. The right panels illustrate the discrepancy $d = u - v$ and its interpolated counterpart $\tilde{d}$ for representative cases. 
For $N_s =24 $, neither AOT nor IDDA converges, as $d - \tilde{d} = O(1)$. For $N_s = 48$, both methods yield nearly identical reconstructions ($d \approx \tilde{d}$), with IDDA exhibiting smaller residual oscillations. 
}
\label{Fig_ns_comp}
\end{subfigure}
\caption{Comparison for KS equation.}
\label{Fig:KS_comp}
\end{figure}

\subsection{Two-Dimensional Navier--Stokes Equation}
\label{sec:NS}

We consider the two dimensional incompressible Navier--Stokes equations in vorticity--streamfunction formulation,
$$
\pderiv{\omega}{t} = -\mathbf{u}\cdot\nabla\omega + \mu\nabla^2\omega, 
\quad \nabla^2\Psi = -\omega, 
\quad \mathbf{u} = \left(\pderiv{\Psi}{y}, -\pderiv{\Psi}{x}\right),
$$
where $\omega$ is the vorticity, $\Psi$ is the streamfunction, and $\mathbf{u}$ is the velocity field.  
The reference initial vorticity field is defined as
$$\begin{aligned}
 \omega(x,y,0) =
50\,&e^{-\tfrac{(x-\tfrac{5\pi}{4})^2+(y-\pi)^2}{0.4}}
-50\,e^{-\tfrac{(x-\tfrac{3\pi}{4})^2+(y-\pi)^2}{0.8}}\\
&+50\,e^{-\tfrac{(x-\pi)^2+(y-\tfrac{3\pi}{2})^2}{0.4}}
-50\,e^{-\tfrac{(x-\pi)^2+(y-\tfrac{\pi}{2})^2}{0.8}}.   
\end{aligned}
$$
The initial condition consists of superpositions of Gaussian vortices, producing complex vortical structures characteristic of two-dimensional turbulence.  
This configuration serves as a stringent test for assessing IDDA's performance in tracking multi-scale, nonlinear flow dynamics such as vortex merging and the inverse energy cascade.

Within the IDDA framework, the nonlinear advection term $-\mathbf{u}\cdot\nabla\omega$ is treated as the non-diffusive operator $\mcf$, while the Laplacian term $\mu\Delta\omega$ represents the dissipative operator $\mcd$.  
To stabilize the interpolated discrepancy under weak physical diffusion ($\mu = 10^{-4}$), we add an artificial diffusion term $\eta\Delta\tilde{d}$ with $\eta = kh$, $k\in[0,2]$, which acts only on the interpolated discrepancy without over-smoothing the assimilated field.
The AOT and IDDA formulations are therefore
\begin{align}
\label{eq:NS_AOT_labelled}
&\text{AOT:}\left\{
\begin{aligned}
  &\pderiv{\zeta}{t} = -\mathbf{u}_\zeta \cdot \nabla \zeta + \mu\laplace\zeta + \lambda \sdtd,\\
  &\laplace\Psi_\zeta = -\zeta, \quad \mathbf{u}_\zeta = (\pderiv{\Psi_\zeta}{y}, -\pderiv{\Psi_\zeta}{x}),
\end{aligned}
\right.\\
\label{eq:NS_IDDA_labelled}
&\text{IDDA:}\left\{
\begin{aligned}
  &\pderiv{\zeta}{t} = -\mathbf{u}_{\zeta+\sdtd} \cdot \nabla (\zeta+\sdtd) + \mu\laplace\zeta + \lambda \sdtd - \eta\laplace\sdtd,\\
  &\laplace\Psi_{\zeta+\sdtd} = -(\zeta+\sdtd), \quad \mathbf{u}_{\zeta+\sdtd} = (\pderiv{\Psi_{\zeta+\sdtd}}{y}, -\pderiv{\Psi_{\zeta+\sdtd}}{x}).
\end{aligned}
\right.
\end{align}

The computational domain is $[0,2\pi]^2$ with periodic boundary conditions.  
Spatial discretization uses $256\times256$ grid points and pseudo-spectral differentiation.  
Observations are sampled at $N_s$ quasi-random locations generated using Halton sequences to ensure uniform spatial coverage.  
Interpolation of observational data is performed using a $C^2$-smooth radial basis function (RBF) interpolant with support radius $r=\rho h$, $\rho\in[1,10]$, where $h$ represents the characteristic spacing between observation points ($\approx2\pi/\sqrt{N_s}$). This choice ensures globally smooth reconstruction while preserving locality and stability of the interpolated discrepancy. The artificial diffusion parameter is $\eta=kh$, $k\in[0,2]$. The nudging parameter is set to $\lambda = 2$, and the assimilated initial vorticity field is $\zeta(x,y,0) = 0$.

Figure \ref{fig:ExNS_comp} summarizes the results of the two-dimensional Navier–Stokes assimilation test using $N_s = 400$ quasi-random observations reconstructed by a compactly supported RBF interpolant with radius $\rho = 5h$ and artificial diffusion $\eta = h$.  Panel (a) compares vorticity fields at $t = 2$ and $t = 6$, showing that the IDDA solution rapidly synchronizes with the reference flow, accurately reproducing the position and strength of the dominant vortices, whereas AOT exhibits visible phase and amplitude errors.  Panel (b) presents the corresponding spatial error distributions and time evolution of the $L^2$ norm of the vorticity difference.  IDDA maintains smaller, more homogeneous errors throughout the domain and achieves exponential convergence with rate $\gamma \approx 2.1 \approx \lambda$, in close agreement with the theoretical prediction, while AOT converges more slowly with $\gamma \approx 0.83$.  These results confirm that incorporating the interpolated discrepancy directly into the nonlinear advection term substantially improves both accuracy and convergence speed, even under sparse observational coverage.

Figure~\ref{ExNS:comp_5} summarizes the parametric convergence behavior of IDDA and AOT for the two-dimensional Navier–Stokes equations.
The left panel demonstrates that IDDA reaches the target convergence rate $\gamma \approx 2\lambda$ once the number of observation points exceeds $N_s \approx 400$,
whereas AOT converges more slowly with the same observation density.
The middle panel shows that IDDA’s performance remains robust across a broad range of RBF support radii ($\rho \ge 2$),
indicating that the method is not overly sensitive to interpolation width.
The right panel highlights the stabilizing effect of artificial diffusion $\eta = kh$, where IDDA maintains near-optimal convergence even for small~$\eta$,
while AOT requires significantly larger diffusion to achieve comparable convergence rates.
Together, these results confirm that IDDA delivers rapid and reliable convergence over a wide range of spatial and numerical parameters.

\begin{figure}[htbp]
\begin{subfigure}{\textwidth}
\centering
\includegraphics[trim= 7.2cm 2.5cm 6cm 1.5cm,clip=true,width=0.95\textwidth]{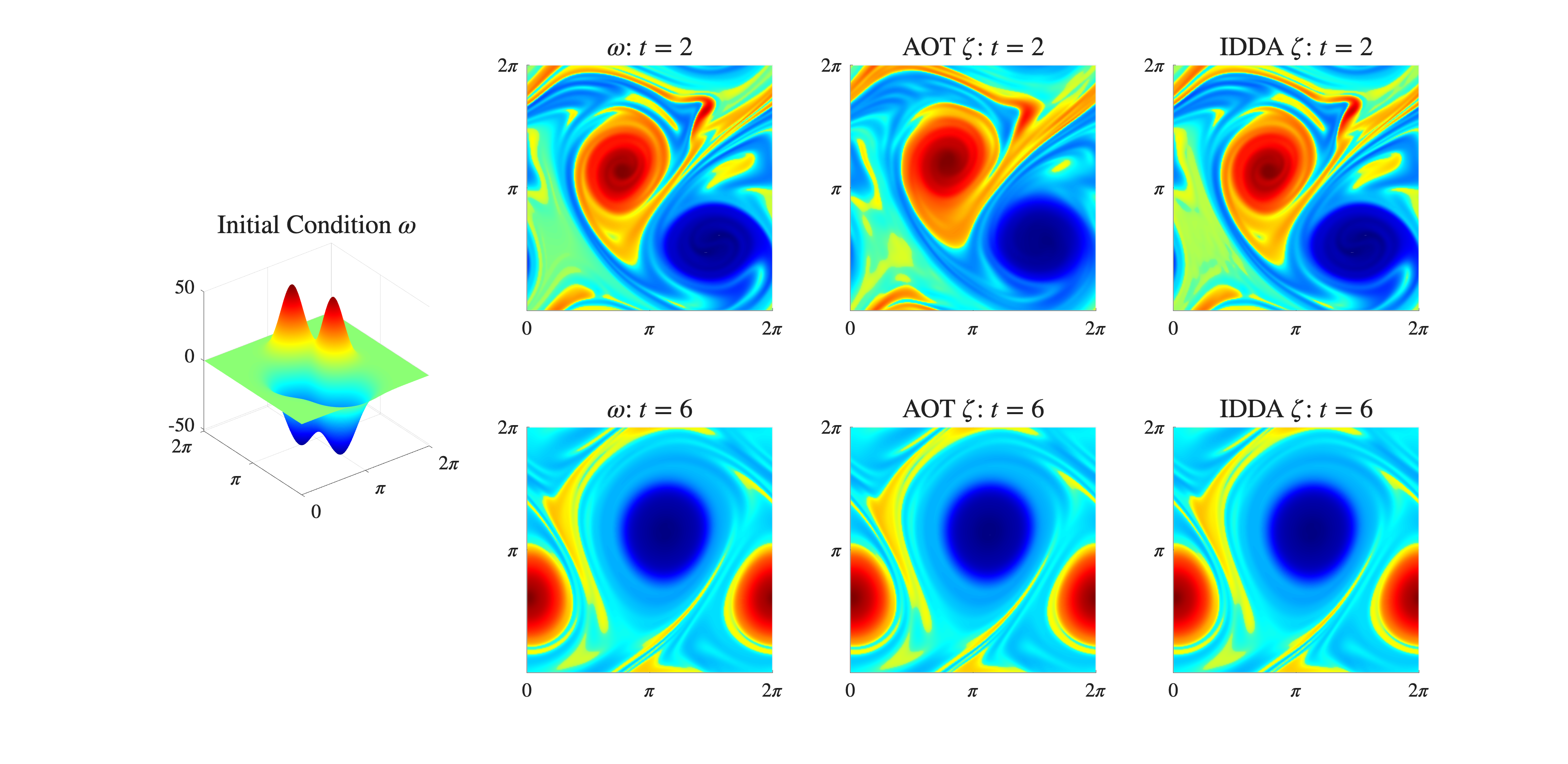}
\caption{Data assimilation for the two-dimensional Navier--Stokes equations in vorticity form. 
Comparison of AOT and IDDA reconstructions at $t = 2$ and $t = 6$ using $N_s = 400$ quasi-random observations, 
RBF interpolation with radius $\rho = 5h$, and artificial diffusion $\eta = h$. 
IDDA rapidly aligns with the reference vorticity field, accurately capturing vortex evolution and merger, 
whereas AOT exhibits noticeable phase and amplitude discrepancies.}
\label{ExNS:comp_1}
\end{subfigure}

\vspace{1em}

\begin{subfigure}{\textwidth}
\centering
\includegraphics[trim= 5.5cm 7cm 6cm 6cm,clip=true,width=0.95\textwidth]{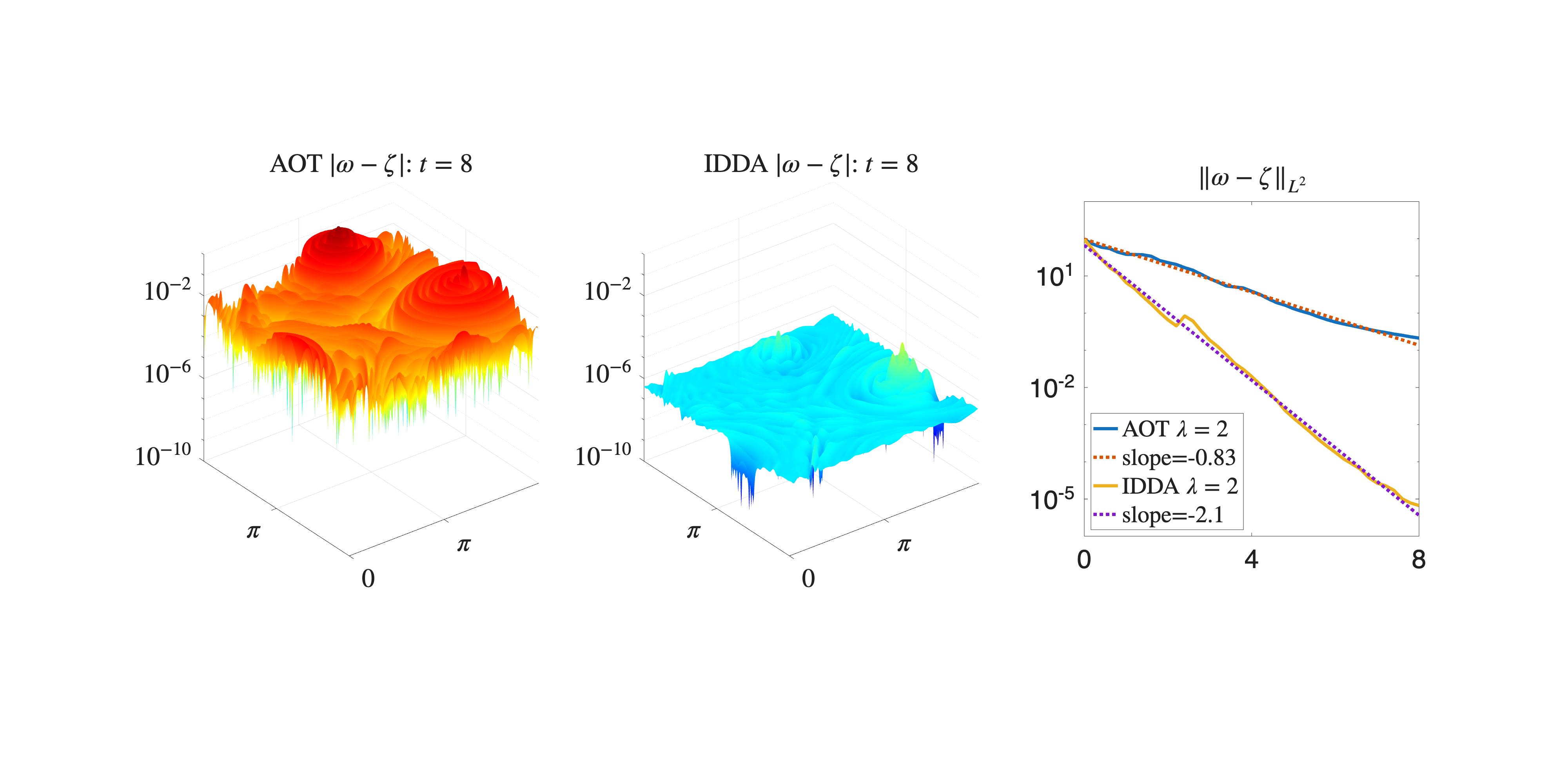}
\caption{Pointwise vorticity errors $|\omega - \zeta|$ at $t = 8$ (left) and time evolution of the $L^2$ error norm (right). 
Spatial error maps show that IDDA maintains smaller, more uniformly distributed errors across the domain, 
while AOT produces larger localized errors near vortex boundaries. 
The time series confirms exponential convergence, with IDDA achieving rate $\approx 2.1$ (matching the target $\lambda$) 
compared to AOT's slower rate $\approx 0.83$.}
\label{ExNS:comp_2}
\end{subfigure}
\caption{Performance comparison of IDDA and AOT for the two-dimensional Navier--Stokes equations with $N_s = 400$ observations, $\lambda = 2$, and artificial diffusion $\eta = h$. 
Panel~(a) shows vorticity field reconstructions demonstrating IDDA’s superior alignment with the reference solution. 
Panel~(b) quantifies this improvement through spatial error distributions and convergence rates.
}
\label{fig:ExNS_comp}
\end{figure}

\begin{figure}[htbp]
\centering
\includegraphics[trim= 0cm 0cm 0cm 0cm,clip=true,width=0.95\textwidth]{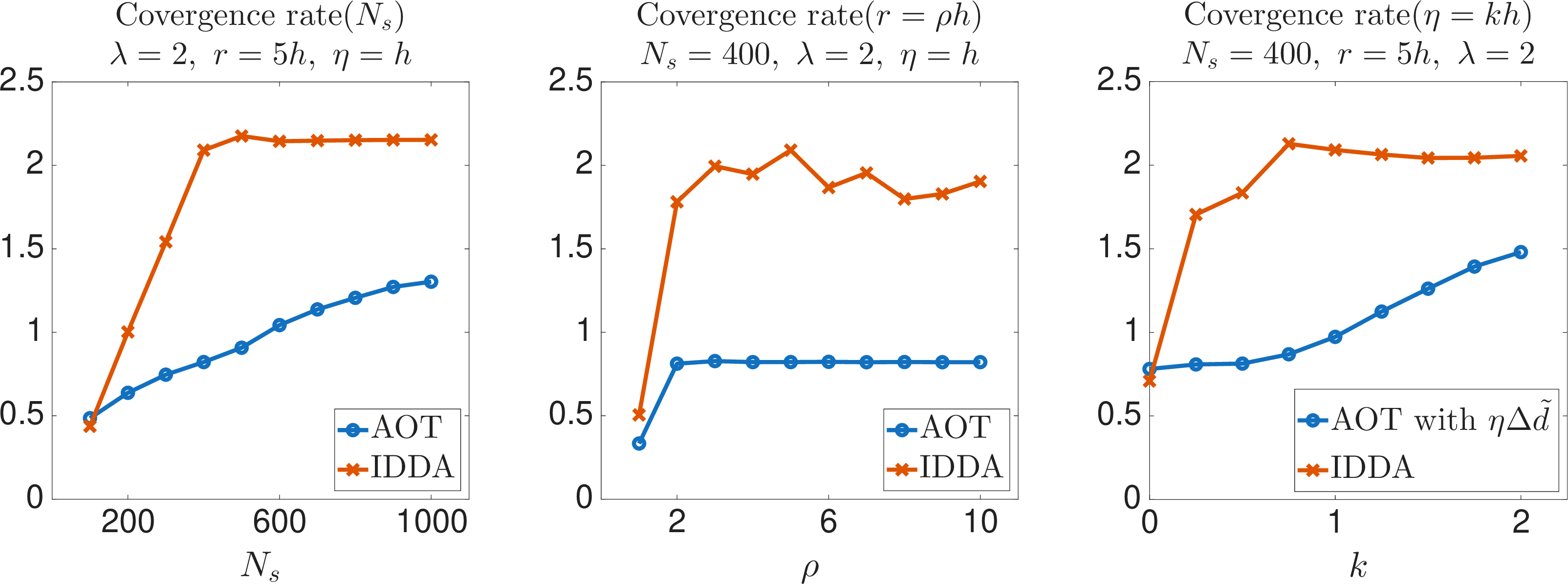}
\caption{Convergence behavior for the two-dimensional Navier--Stokes equations with varying observation density, interpolation radius, and artificial diffusion. 
The left panel shows convergence rates versus the number of observation points $N_s$ (ranging from 100 to 1000) for both AOT and IDDA with $\lambda = 2$, $\rho = 5h$, and $\eta = h$. 
IDDA achieves rates close to the theoretical target when $N_s \ge 400$. 
The middle panel shows the dependence of convergence rate on the RBF support radius $r=\rho h$ (ranging from $h$ to $10h$); IDDA consistently reaches the target rate for $\rho \ge 2$. 
The right panel examines the effect of the artificial diffusion coefficient $\eta = kh$  (ranging from $0$ to $2h$) on convergence, where AOT is augmented with matching artificial diffusion for a fair comparison. 
Across all tested parameters, IDDA maintains faster and stable convergence than AOT.
}
\label{ExNS:comp_5}
\end{figure}

\section{Limitations and Future Directions}
\label{sec:limitations}

While IDDA demonstrates clear improvements over interpolated AOT in both theoretical analysis and numerical performance, several limitations and open problems remain that merit further study.

\subsection{Theoretical Limitations}

The convergence conditions established in Theorem~\ref{thm:IDDA_convergence} are sufficient but not necessarily sharp. 
The interpolation coercivity parameter $\alpha$ in Lemma~\ref{lem:interp_coercivity} provides a conservative lower bound; numerical evidence indicates that effective values can approach one for smooth solutions and accurate interpolation. 
A refined analysis connecting $\alpha$ to solution regularity, interpolation order, and grid spacing could yield sharper and more practical bounds on the required nudging strength $\lambda$.

The current proof framework assumes solutions with bounded spatial gradients. 
For PDEs exhibiting discontinuities or steep gradients (e.g., Burgers' shocks), the interpolation error bound $\|f - \tilde{f}\| \le Ch\|\nabla f\|$ no longer holds. 
Extending IDDA to handle limited-regularity solutions may require adaptive observations, locally refined interpolation, or shock-aware reconstruction schemes.

\subsection{Practical Considerations}

All tests used uniformly spaced or quasi uniform observation networks. Real systems often have highly irregular spatial sampling. Investigating the effect of observation clustering, adaptive density, and anisotropic coverage on IDDA performance would improve its practical applicability.

Selecting appropriate assimilation parameters is essential for stable and efficient IDDA implementation. 
Numerical experiments indicate that moderate nudging stren\-gths ($\lambda = 1$--$4$) generally yield reliable convergence, 
though optimal values may vary in time or space depending on local dynamics. 
The artificial diffusion parameter $\eta$ (Remark~\ref{rem:art_diss}) likewise requires empirical tuning. 
Recent work by \v{C}ib\'{\i}k et al.~\cite{cibik2025adaptive} introduced adaptive nudging algorithms that dynamically adjust $\lambda$ 
based on observed residuals, achieving effective values orders of magnitude smaller than theoretical worst-case estimates. 
Such residual-based strategies could be integrated into IDDA by monitoring the mismatch between model predictions and interpolated observations, 
allowing $\lambda$ to evolve automatically in time to maintain target convergence rates. 
An alternative framework proposed by Diegel et al.~\cite{diegel2025analysis} employs large nudging parameters with implicit time-stepping to achieve 
optimal accuracy independent of $\lambda$, at the expense of higher computational cost for nonlinear solvers. 
IDDA's robustness with moderate $\lambda$ and explicit schemes offers a complementary approach, balancing efficiency and simplicity while 
remaining compatible with adaptive parameter selection techniques.

In practical applications, observational and modeling imperfections present significant challenges to data assimilation. 
Real measurements inevitably contain noise, with observed data of the form 
$u_{\text{obs}}(x_k,t) = u(x_k,t) + \epsilon_k(t)$, where $\epsilon_k$ represents measurement error. 
Such noise propagates through interpolation as $\tilde{u}_{\text{noisy}} = \tilde{u} + \tilde{\epsilon}$, 
affecting both the nudging term $\lambda\tilde{d}$ and the nonlinear operator $\mcf[v+\tilde{d}]$. 
Preliminary analysis suggests that bounded measurement noise $\|\epsilon_k\|\le\sigma$ leads to a bounded steady-state error rather than exponential convergence to zero. 
Developing noise-robust variants of IDDA is therefore an important future direction.

Model inaccuracies introduce an additional source of error. 
When the true state satisfies $\partial_t u = \mcg[u] + \delta(x,t)$ for some model discrepancy $\delta$, 
the IDDA error evolves as
$$
\partial_t d = \mcd[d] + (\mcf[u] - \mcf[v+\tilde{d}]) - \lambda\tilde{d} + \delta.
$$
Persistent model errors $\delta\neq0$ prevent exact synchronization, leading to nonzero asymptotic bias. 
Addressing this limitation may require adaptive or time-dependent nudging parameters $\lambda(t)$ that respond to observed model--data mismatches, 
or the incorporation of statistical error models to compensate for structural discrepancies in $\mcg$.

Finally, many real systems are multiscale or coupled, involving partial observability, scale separation, or stiffness. 
Examples include atmospheric and oceanic models where only subsets of variables are observed, 
coarse-grained satellite measurements coupled with high-resolution simulations, or reacting flows combining fast and slow processes. 
Extending IDDA to these settings will require careful treatment of the operator decomposition $\mcg = \mcf + \mcd$, 
potentially integrating multiscale interpolation and implicit or semi-implicit time-stepping strategies.

\subsection{Extensions and Future Research}

Future work can extend IDDA in several directions. Incorporating filtering or regularization techniques would improve robustness under noisy or uncertain observations, where measurement errors propagate through interpolation. Adaptive or stochastic nudging parameters could help mitigate model error and imperfect dynamics when the underlying PDE model is approximate. Automated tuning of $\lambda$ and artificial diffusion $\eta$, following residual-based adaptive strategies~\cite{cibik2025adaptive}, could make IDDA self-regulating without manual calibration. Extending the method to coupled or multiscale systems, such as reacting flows or geophysical models, will require coupling-aware interpolation and possibly implicit or semi-implicit schemes. Finally, testing IDDA in large-scale forecasting or ocean circulation models will assess its performance under realistic, sparse, and noisy observation networks.

\section{Summary and Conclusions}
\label{sec:conclusion}

We introduced the Interpolated Discrepancy Data Assimilation (IDDA) method, a modification of continuous data assimilation designed for sparse physical-space observations. By incorporating the interpolated discrepancy both as a nudging term and within the nonlinear operator, IDDA aligns model dynamics directly with observed behavior, achieving robust and predictable convergence without requiring large nudging parameters. Theoretical analysis established exponential error decay under explicit conditions on observation density, diffusion, and feedback strength, while numerical experiments across viscous Burgers', KPP--Burgers', Kuramoto--Sivashinsky, and Navier--Stokes equations confirmed consistent agreement with theoretical predictions and superior performance compared with interpolated AOT. Future work will focus on extending IDDA to handle noisy observations, model error, coupled or multiscale systems, and adaptive parameter selection, further bridging the gap between mathematical theory and operational data assimilation practice.

\section*{Acknowledgements}

The authors gratefully acknowledge support from Laboratory Directed Research and Development program at Los Alamos National Laboratory under project number 20260103ER. We thank Adam Larios and Elizabeth Carlson for valuable discussions on continuous data assimilation methods and convergence analysis. We are particularly grateful to Edriss Titi for insightful comments on interpolation-based data assimilation. Claude (Anthropic AI assistant) was used to improve the organization of the manuscript, enhance the clarity of the narrative, and improve the quality of the prose. All scientific content, analyses, simulations, and interpretations remain the sole responsibility of the authors.

\section*{Disclosure Statement}

The authors declare no conflict of interest.

\textbf{Author Contributions:} T.W. developed the IDDA method, performed the rigorous convergence analysis (Theorem~\ref{thm:IDDA_convergence} and Lemma~\ref{lem:interp_coercivity}), and conducted numerical experiments. H.G. contributed to algorithm design and provided oceanographic application context. V.G. assisted with mathematical analysis and computational implementation. J.M.H. supervised the research and provided guidance on data assimilation theory. All authors contributed to manuscript preparation.

\appendix

\section{Convergence Analysis for Interpolated AOT}
\label{sec:AOT_appendix}

For completeness and to facilitate direct comparison with IDDA (Theorem~\ref{thm:IDDA_convergence}), we provide a rigorous convergence proof for the interpolated Azouani-Olson-Titi approach under identical assumptions. This analysis parallels the IDDA proof but reveals the fundamental difference: AOT's convergence condition depends explicitly on the Lipschitz constant $L$, whereas IDDA's depends on the interpolation error scale $h^2$.

\begin{theorem}[AOT Exponential Convergence]\label{thm:AOT_convergence}
Let $\Omega\subset\mathbb{R}^d$ be a bounded domain with periodic or homogeneous Dirichlet boundary conditions, and consider the dissipative PDE
$$
u_t = \mcf[u]+\mcd[u],
$$
where $u(\mathbf{x},t)$ is the reference solution. Assume that Assumptions~\ref{assump:diffusion}--\ref{assump:observations} hold. Let $v(\mathbf{x},t)$ be the assimilated solution satisfying the interpolated AOT dynamics
\begin{equation}\label{eq:AOT_dynamics_thm}
v_t = \mcf[v] + \mcd[v] + \lambda\,\sdtd, 
\qquad d := u-v.
\end{equation}
If the nudging parameter $\lambda$ and observation spacing $h$ satisfy
\begin{equation}\label{eq:AOT_cond_thm}
\frac{L}{\alpha} < \lambda < \frac{2\mu}{C^2h^2},
\end{equation}
where we use the coercivity constant $\alpha = 1/2$ from Lemma~\ref{lem:interp_coercivity} (giving the factor of 2), then the $L^2$ error norm decays exponentially:
\begin{equation}\label{eq:AOT_decay_thm}
    \|d(t)\| \le e^{-\gamma_{\text{AOT}} t}\,\|d(0)\|,
    \qquad
    \gamma_{\text{AOT}} = \lambda\alpha - L > 0.
\end{equation}
\end{theorem}

\begin{proof}[Proof of Theorem~\ref{thm:AOT_convergence}]
Subtracting the true equation from the AOT dynamics \eqref{eq:AOT_dynamics_thm} yields:
\begin{equation}\label{eq:AOT_error_evolution}
\partial_t d = \mcd[d] + (\mcf[u] - \mcf[v]) - \lambda\,\sdtd.
\end{equation}

Taking the $L^2$ inner product with $d$:
\begin{equation}\label{eq:AOT_energy}
\frac{1}{2}\frac{d}{dt}\|d\|^2
= \langle \mcd[d],d\rangle + \langle \mcf[u]-\mcf[v], d\rangle - \lambda \langle \sdtd, d\rangle.
\end{equation}

By Assumption~\ref{assump:diffusion},
\begin{equation}\label{eq:AOT_dissipative}
\langle \mcd[d], d\rangle \le -\mu\|\nabla d\|^2.
\end{equation}

This is the key difference from IDDA. Using the Lipschitz condition directly on the error $d$:
\begin{equation}\label{eq:AOT_nonlinear}
\langle \mcf[u]-\mcf[v], d\rangle
\le \|\mcf[u]-\mcf[v]\|\|d\|
\le L\|d\|^2.
\end{equation}
Note the crucial difference: this bound is proportional to $\|d\|^2$, not to $\|\nabla d\|\|d\|$ as in IDDA.

Using Lemma~\ref{lem:interp_coercivity}:
\begin{equation}\label{eq:AOT_nudging}
-\,\lambda\langle \sdtd, d\rangle
\le -\,\lambda\alpha\|d\|^2 + \frac{\lambda C^2h^2}{2} \|\nabla d\|^2.
\end{equation}

\textbf{Step 4: Combining terms.} Substituting \eqref{eq:AOT_dissipative}, \eqref{eq:AOT_nonlinear}, and \eqref{eq:AOT_nudging} into \eqref{eq:AOT_energy}:
\begin{equation}\label{eq:AOT_combined}
\frac{1}{2}\frac{d}{dt}\|d\|^2
\le -\mu\|\nabla d\|^2
+ L\|d\|^2
- {\lambda\alpha}\|d\|^2
+ \frac{\lambda C^2h^2}{2} \|\nabla d\|^2.
\end{equation}

Rearranging:
\begin{equation}\label{eq:AOT_final_energy}
\frac{1}{2}\frac{d}{dt}\|d\|^2
\le - \left(\lambda\alpha - L\right)\|d\|^2
- \left(\mu - \frac{\lambda C^2 h^2}{2}\right)\|\nabla d\|^2.
\end{equation}

Under condition \eqref{eq:AOT_cond_thm}, both coefficients are strictly positive. Define
$$
\gamma_{\text{AOT}} := \lambda\alpha - L > 0.
$$
Discarding the non-negative gradient term:
\begin{equation}\label{eq:AOT_gronwall}
\frac{d}{dt}\|d\|^2 \le -2\gamma_{\text{AOT}} \|d\|^2.
\end{equation}

Applying Gr\"onwall's inequality gives \eqref{eq:AOT_decay_thm}.
\end{proof}

\bibliographystyle{siamplain}
\bibliography{IDDA}

\end{document}